\documentclass[reqno, a4paper, 11pt]{amsart}

\usepackage{amssymb,amsmath,mathtools,amsthm,bbm}
\usepackage{enumerate}
\usepackage{xcolor}
\usepackage{mathrsfs}
\usepackage{enumitem}
\usepackage{color}

\usepackage{hyperref}  % pour les hyperliens (table des matières notamment)
\hypersetup{                    % parametrage des hyperliens
	colorlinks=true,                % colorise les liens
	breaklinks=true,                % permet les retours à la ligne pour les liens trop longs
	urlcolor= blue,                 % couleur des hyperliens
	linkcolor= black,                % couleur des liens internes aux documents (index, figures, tableaux, equations,...)
	citecolor= magenta,                % couleur des liens vers les references bibliographiques
	pdfstartview = FitH,	bookmarksopen = true               % marques-page déroulés
}

\usepackage{txfonts}
\usepackage{newtxtext}
\usepackage{tikz}
\usepackage{caption}
\usepackage{subcaption}
\usepackage{accents}
\usepackage{cleveref}

\newcommand{\W}{\mathcal{W}}

\newcommand{\R}{\mathbb{R}}

\newcommand{\E}{\mathbb{E}}
\newcommand{\PP}{\mathbb{P}}

\newcommand{\tmu}{\tilde{\mu}}
\newcommand{\tk}{\tilde{k}}
\newcommand{\tx}{\tilde{x}}
\newcommand{\ty}{\tilde{y}}
\newcommand{\tnu}{{\tilde{\nu}}}
\newcommand{\tchi}{\tilde{\chi}}
\newcommand{\tpi}{\tilde{\pi}}

\newcommand{\tsigma}{{\tilde{\sigma}}}
\newcommand{\cW}{\mathcal{W}}
\newcommand{\bW}{\overline{\mathcal{W}}_2}
\newcommand{\cP}{\mathcal{P}}
\newcommand{\cI}{\mathcal{I}}
\newcommand{\cJ}{\mathcal{J}}
\newcommand{\cN}{\mathcal{N}}
\newcommand{\cS}{{\mathcal S}}
\newcommand{\cO}{{\mathcal O}}
\newcommand{\cB}{{\mathcal B}}
\newcommand{\tr}{{\rm tr}}
\newcommand{\BW}{{\bf bw}}

\DeclareMathOperator{\co}{\operatorname{co}}

\renewcommand{\epsilon}{\varepsilon}

\numberwithin{equation}{section}
\theoremstyle{plain}
\newtheorem{prooff}{Proof}[section]

\newtheorem{lemma}[prooff]{Lemma}
\newtheorem{theorem}[prooff]{Theorem}
\newtheorem{proposition}[prooff]{Proposition}
\newtheorem{corollary}[prooff]{Corollary}

\theoremstyle{definition}
\newtheorem{example}[prooff]{Example}
\newtheorem{remark}[prooff]{Remark}

\begin{document}

\title[Wasserstein projection in the convex order]{Wasserstein projections in the convex order: regularity and characterization in the quadratic Gaussian case}

\author{Aur\'elien Alfonsi and Benjamin Jourdain%\footnotemark[1]
}
\thanks{CERMICS, ENPC, Institut Polytechnique de Paris, INRIA, Marne-la-Vall\'ee, France. E-mails: aurelien.alfonsi@enpc.fr, benjamin.jourdain@enpc.fr. This research benefited from the support of the chair Risques Financiers, Fondation du Risque.}
\maketitle
\begin{abstract} 
  In this paper, we first show continuity of both Wasserstein projections in the convex order when they are unique. We also check that, in arbitrary dimension $d$, the quadratic Wasserstein projection of a probability measure $\mu$ on the set of probability measures dominated by $\nu$ in the convex order is non-expansive in $\mu$ and H\"older continuous with exponent $1/2$ in $\nu$. When $\mu$ and $\nu$ are Gaussian,  we check that this projection is Gaussian and also consider the quadratic Wasserstein projection on the set of probability measures $\nu$ dominating $\mu$ in the convex order. In the case when $d\ge 2$ and $\nu$ is not absolutely continuous with respect to the Lebesgue measure where uniqueness of the latter projection was not known, we check that there is always a unique Gaussian projection and characterize when non Gaussian projections with the same covariance matrix also exist. Still for Gaussian distributions, we characterize the covariance matrices of the two projections. It turns out that there exists an orthogonal transformation of space under which the computations are similar to the easy case when the covariance matrices of $\mu$ and $\nu$ are diagonal.

\end{abstract}

For $\rho\ge 1$, we denote the celebrated $\rho$-Wasserstein distance between $\mu$ and $\nu$ in the set $\mathcal P_\rho(\R^d)$ of probability measures on $\R^d$ with  finite $\rho$-th moment by
\begin{equation}
	\label{eq:def.Wasserstein_distance}
	\textstyle
	\W_\rho(\mu,\nu)= \left(\inf_{ \pi \in \Pi(\mu,\nu)} \int_{\R^d \times \R^d} |x - y|^p \pi(dx,dy)\right)^{1/p},
\end{equation} 
where we write $\Pi(\mu,\nu)\subset\cP_\rho(\R^d\times\R^d)$ for the set of couplings with marginals $\mu$ and $\nu$. We say that $\mu$ is smaller than $\nu$ in the convex order and denote $\mu \le_{cx} \nu$ if 
\begin{equation}
	\label{eq:def.convex_order}
	\textstyle\forall f\colon \R^d \to \R\mbox{ convex},\;
	\int_{\R^d} f(x) \, \mu(dx) \le \int_{\R^d} f(x) \, \nu(dy).
\end{equation}
We denote the metric projections with respect to $\mathcal W_\rho$ of $\mu$ (resp.\ $\nu$) onto $\{ \eta \in \mathcal P_\rho(\R^d) : \eta \le_{cx} \nu \}$ (resp.\ $\{ \eta \in \mathcal P_\rho(\R^d) : \mu \le_{cx} \eta \}$) by $\cI_\rho(\mu,\nu)$ and $\cJ_\rho(\nu,\mu)$ :
$$\cW_\rho(\mu,\cI_\rho(\mu,\nu))=\inf_{\eta\in\cP_\rho(\R^d):\eta\le_{cx}\nu}\cW_\rho(\mu,\eta)\mbox{ and }\cW_\rho(\nu,\cJ_\rho(\nu,\mu))=\inf_{\eta\in\cP_\rho(\R^d):\mu\le_{cx}\eta}\cW_\rho(\nu,\eta).$$
These Wasserstein projections in the convex order were introduced in \cite{AlCoJo20} to restore the convex order between finitely supported approximations of two ordered probability measures. Minimizers always exist and, when $\rho>1$, % the first projection
$\cI_\rho(\mu,\nu)$ is unique while uniqueness of $\cJ_\rho(\nu,\mu)$ was only established when $d=1$ or $\nu$ is absolutely continuous with respect to the Lebesgue measure. Moreover, according to \cite[Corollary 4.4]{AlCoJo20},
\begin{equation}
   \label{fb}\forall \mu,\nu\in\cP_\rho(\R^d),\;\cW_\rho(\mu,\cI_\rho(\mu,\nu))=\cW_\rho(\nu,\cJ_\rho(\nu,\mu)).
\end{equation}
When $d = 1$, the projections are denoted by $\cI(\mu,\nu)$ and $\cJ(\nu,\mu)$ since they do not depend on $\rho>1$ and explicit formulas for their quantile functions are available. The quantile function of $\eta\in\cP(\R)$ is denoted by $F_\eta^{-1}(u)=\inf\{x\in\R:\eta((-\infty,x])\ge u\},\;u\in(0,1)$. Then \cite[Theorem 2.6 and Proposition 4.2]{AlCoJo20} state, for $u \in (0,1)$,
\begin{equation}\label{eq:quantileProjections}
F_{\cI(\mu,\nu)}^{-1}(u)=F_\mu^{-1}(u)-\partial_-\co(G)(u)\text{ and } F_{\cJ(\nu,\mu)}^{-1}(u)=F_\nu^{-1}(u)+\partial_-\co(G)(u),
\end{equation}
where $G(u) := \int_0^u(F_\mu^{-1}-F_\nu^{-1})(v)\,dv$, $\co$ denotes the convex hull, and $\partial_-$ the left-hand derivative.

Dual formulations of the minimization problems defining $\mathcal I_2(\mu,\nu)$ and $\mathcal J_2(\mu,\nu)$ have been studied by Kim and Ruan \cite{KimRuan} who in particular show in Corollary 8.5 that when $\mu,\nu\in\cP_2(\R^d)$ with $\nu$ absolutely continuous with respect to the Lebesgue measure then there exists a continuously differentiable convex function $\phi:\R^d\to\R$ such that $\cI_2(\nu,\mu)=\nabla \phi \#\mu$ and $\nu=\nabla \phi \# \cJ_2(\nu,\mu)$, thus generalizing the one-dimensional statement of \cite[Proposition 4.5]{AlCoJo20}.

Gozlan, Roberto, Samson, and Tetali \cite{GoRoSaTe17} introduced a related generalization of optimal transport, the weak optimal transport, in order to study measure concentration inequalities.
The following barycentric weak optimal transport problem received in recent years special attention, see for example \cite{GoRoSaTe17,GoRoSaSh18,GoJu18,AlCoJo19,AlCoJo20,BaBePa18,BaBePa19}:
for $\mu, \nu \in \mathcal P_\rho(\R^d)$, consider
\begin{equation}\label{eq:def_WOT}
	\textstyle
	\mathcal V_\rho^\rho(\mu,\nu) := \inf_{\pi \in \Pi(\mu,\nu)} \int_{\R^d} \left| x - \int_{\R^d} y \, \pi_x(dy) \right|^\rho \, \mu(dx),
\end{equation}
where we write $(\pi_x)_{x \in \R^d}$ for a disintegration kernel of $\pi$ w.r.t.\ its $\mu$-marginal: $\pi(dx,dy)=\mu(dx)\pi_x(dy)$.
This % barycentric weak optimal transport
problem has an intrinsic connection with the problem of finding Wasserstein projection.
Indeed, we have that the values of $\mathcal V_\rho(\mu,\nu)$ and $\W_\rho(\mu,\mathcal I_\rho(\mu,\nu))$ coincide, see \cite{AlCoJo20,BaBePa18}.
Moreover, if $\pi^\ast$ is an optimizer of \eqref{eq:def_WOT} then the image of the first marginal $\mu$ under the map $x \mapsto \int_{\R^d} \pi_x^\ast(y)\,dy$ is a minimizer of $\inf_{\eta\le_c\nu}{\mathcal W}_\rho(\mu,\eta)$ and coincides with $\mathcal I_\rho(\mu,\nu)$ when $\rho > 1$.
Therefore, when $\mu, \nu \in \mathcal P_\rho(\R^d)$ are finitely supported, \eqref{eq:def_WOT} can be used to compute the Wasserstein projection.
In particular, ${\mathcal I}_2(\mu,\nu)$ can be computed by solving a quadratic optimization problem with linear constraints.
We refer to \cite{GoRoSaTe17,AlBoCh18} for dual formulations of weak optimal transport problems with additional martingale constraints, to \cite{BaBePa18} for the existence of optimal couplings and necessary and sufficient optimality conditions, to \cite{BaPa19} for continuity of their value function in terms of the marginal distributions $\mu$ and $\nu$, and to \cite{BaPa20} for applications of such problems.
We point out the connection of Wasserstein projections to Caffarelli's contraction theorem that was discovered in \cite{FaGoPr20}.

The following Lipschitz regularity of the projections was established in \cite[Propositions 3.1 and 4.3]{AlCoJo20} under the assumption $\mu\le_{cx}\nu$ (so that $\mathcal I_\rho(\mu,\nu)=\mu$ and $\mathcal J_\rho(\nu,\mu)=\nu$) :
\begin{align}
   \mathcal W_\rho(\cI_\rho(\mu,\nu),\cI_\rho(\tilde \mu,\tilde \nu))&\le 2\mathcal W_\rho(\mu,\tilde \mu)+\phantom{2}\mathcal W_\rho(\nu,\tilde \nu)\label{lipproi}\\\mathcal W_\rho(\cJ_\rho(\nu,\mu),\cJ_\rho(\tilde \nu,\tilde \mu))&\le\phantom{2}\mathcal W_\rho(\mu,\tilde \mu)+2\mathcal W_\rho(\nu,\tilde \nu).\label{lipproj}
\end{align}
In dimension $d=1$, generalization of those estimations without convex ordering of $\mu$ and $\nu$ and with $\cI_\rho$ and $\cJ_\rho$ replaced by $\cI$ and $\cJ$ was proved in \cite{JMP23} by exploiting \eqref{eq:quantileProjections}. In \cite{KKRW}, the authors introduce a statistical test for convex order based on the quadratic Wasserstein projection distance and state in Theorem 4.3 and Remark 4.4 {a quadratic estimation which can easily be extended to any $\rho\ge 1$:}  for $\mu,\nu,\tilde\mu,\tilde\nu\in\cP_\rho(\R^d)$, 
\begin{align}
  |\cW_\rho(\mu,\cI_\rho(\mu,\nu))-\cW_\rho(\tilde \mu,\cI_\rho(\tilde \mu,\tilde \nu))|&\le \cW_\rho(\mu,\tilde\mu)+ \cW_\rho(\nu,\tilde\nu)\label{lipdisti}\\
  |\cW_\rho(\nu,\cJ_\rho(\nu,\mu))-\cW_\rho(\tilde \nu,\cJ_\rho(\tilde \nu,\tilde \mu))|&\le \cW_\rho(\mu,\tilde\mu)+ \cW_\rho(\nu,\tilde\nu).\label{lipdistj}
\end{align}
Of course, such a Lipschitz estimation of the projection distance in $\R$ follows from the much stronger matching Lipschitz estimation of the projections in $\cP_\rho(\R^d)$. In the first section of the present paper, we take advantage of the continuity of the projection distance to deduce continuity of the projections when uniqueness holds i.e. under the sole assumption that $\rho>1$ for $\cI_\rho$ and the additional assumption that $d=1$ or the probability measure $\nu$ is absolutely continuous with respect to the Lebesgue measure for $\cJ_\rho$.

In the second section, we prove that in any dimension $d$, $\cP_2(\R^d)\ni\mu\mapsto \mathcal{I}_2(\mu,\nu)$ is non-expansive with respect to $\cW_2$ by a straightforward adaptation of the proof of the same property  for the projection on a closed convex set in an Hilbert space. This is rather remarkable since, while the subset of $\cP_2(\R^d)$ consisting of probability measures dominated by $\nu$ in the convex order is closed and convex, $\cP_2(\R^d)$ is not even a vector space. As a consequence, the factor $2$ in the right-hand side of \eqref{lipproi} is not needed in the quadratic case $\rho=2$ when $\mu\le_{cx}\nu$ or $d=1$. We also check that $\cP_2(\R^d)\ni\nu\mapsto \mathcal{I}_2(\mu,\nu)$ is locally H\"older continuous with exponent $\frac 12$.

The remaining of the paper is devoted to the quadratic Wasserstein projections $\cI_2(\cN_d(m_\mu,\Sigma_\mu),\cN_d(m_\nu,\Sigma_\nu))$ and $\cJ_2(\cN_d(m_\nu,\Sigma_\nu),\cN_d(m_\mu,\Sigma_\mu))$ when $m_\mu,m_\nu\in\R^d$, $\Sigma_\mu,\Sigma_\nu$ belong to the set $\cS_d^+$ of symmetric positive semi-definite $d\times d$ matrices and $\cN_d(m,\Sigma)$ denotes the Gaussian distribution with expectation $m\in\R^d$ and covariance matrix $\Sigma\in\cS^+_d$. In the third section, we check that $\cI_2(\cN_d(m_\mu,\Sigma_\mu),\cN_d(m_\nu,\Sigma_\nu))$ is Gaussian.
We also prove that when $d\ge 2$ and $\Sigma_\nu$ is singular (so that $\cN_d(m_\nu,\Sigma_\nu)$ is not absolutely continuous), then uniqueness of the projection of $\cN_d(m_\nu,\Sigma_\nu)$ on the set of probability measures dominating $\cN_d(m_\mu,\Sigma_\mu)$ may indeed fail and we characterize the failure in terms of the covariance matrices $\Sigma_\mu$ and $\Sigma_\nu$.  But all the projections share the same expectation and the same covariance matrix and there always exists a unique projection among Gaussian distributions that we still denote by $\cJ_2(\cN_d(m_\nu,\Sigma_\nu),\cN_d(m_\mu,\Sigma_\mu))$.

In the fourth section, we characterize the respective covariance matrices $\Sigma_{\cI}$ and $\Sigma_{\cJ}$ of $\cI_2(\cN_d(m_\mu,\Sigma_\mu),\cN_d(m_\nu,\Sigma_\nu))$ and $\cJ_2(\cN_d(m_\nu,\Sigma_\nu),\cN_d(m_\mu,\Sigma_\mu))$ by leveraging on the explicit expression for the quadratic {Wasserstein} distance between two Gaussian distributions derived in \cite{Dowlan,GiSho,Olpuk}.  When $\Sigma_\mu$ and $\Sigma_\nu$ are non singular, there exists an orthogonal matrix $O\in\R^{d\times d}$ such that $ODO^*\Sigma_\mu ODO^*\le \Sigma_\nu $ with $D={\rm diag}\left(1\wedge\sqrt{\frac{(O^*\Sigma_\nu O)_{11}}{(O^*\Sigma_\mu O)_{11}}},\cdots,1\wedge\sqrt{\frac{(O^*\Sigma_\nu O)_{dd}}{(O^*\Sigma_\mu O)_{dd}}}\right)$ and for any such $O$, we have $\Sigma_{\cI}=ODO^*\Sigma_\mu ODO^*$ and $\Sigma_{\cJ}=OD^{-1}O^*\Sigma_\nu OD^{-1}O^*$. The more involved singular case is covered by the general result stated in Theorem \ref{thmprojgauss}.

 {\bf Notation :}

\begin{itemize}
\item For $d\ge 1$, $\cS^+_d$ (resp. $\cS^{+,*}_d$) is the set of $d\times d$ positive semi-definite (resp. positive definite) matrices, $\cO_d$ is the group of $d\times d$ orthogonal matrices, and $\mathfrak{C}_d=\{ \Sigma \in \cS^+_d: \forall i\in \{1,\dots,d\},\ \Sigma_{ii}=1 \}$ is the set of correlation matrices. 
\item For $d\ge 1$ and $\lambda_1,\dots,\lambda_d \in \R$, ${\rm diag}(\lambda_1,\dots,\lambda_d)$ is the diagonal matrix with elements $\lambda_i$, $1\le i\le d$. 
\item For a $d\times d$ matrix $M$, we denote $\textup{dg}(M)={\rm diag}(M_{11},\dots,M_{dd})$ the diagonal matrix made with the diagonal elements of~$M$.  
\item For a $d\times d$ matrix $M$, we note $M_{i,j}$ or simply $M_{ij}$ when there is no ambiguity the element of the $i^{th}$ row and $j^{th}$ column. For
  $1\le i_1\le i_2\le d$ and $1\le j_1\le j_2\le d$, we denote $(M)_{i_1:i_2,j_1:j_2}$ the $(i_2-i_1+1)\times (j_2-j_1+1)$ matrix such that $\left((M)_{i_1:i_2,j_1:j_2}\right)_{ij}=M_{i+i_1-1,j+j_1-1}$ for $1\le i \le i_2-i_1+1$ and $1\le j \le j_2-j_1+1$.
\item For $\eta$ in the set $\cP(\R^d)$ of probability measures on $\R^d$ and $f:\R^d\to\R^{d'}$, we denote by $f\#\eta\in\cP(\R^{d'})$ the image of $\eta$ by $f$.
\item For $\mu,\nu,\eta\in\cP(\R^d)$ and $\pi\in\Pi(\mu,\nu),k\in\Pi(\nu,\eta)$ we denote by $k_y(dz)$ a Markov kernel such that $k(dy,dz)=\nu(dy)k_y(dz)$ and by $\pi k$ the probability measure on $\R^d\times\R^d$ defined by $\pi k(dx,dz)=\int_{y\in\R^d}\pi(dx,dy)k_y(dz)$. Note that $\pi k\in\Pi(\mu,\eta)$ and $\pi k(dx,dz)=\mu(dx)\pi k_x(dz)$. \end{itemize}
\section{Continuity of the projections with index $\rho>1$}
\begin{proposition}\label{propcont}
   Let $\rho>1$. Then $\cP_\rho(\R^d)\times \cP_\rho(\R^d)\ni(\mu,\nu)\mapsto \cI_\rho(\mu,\nu)$ is continuous for $\cW_\rho$ and so is $\cP_\rho(\R)\times \cP_\rho(\R)\ni(\nu,\mu)\mapsto \cJ_\rho(\nu,\mu)$. Last, denoting by $\cP^{\rm abs}_\rho(\R^d)$ the subset of $\cP_\rho(\R^d)$ consisting in probability measures that admit a density with respect to the Lebesgue measure on $\R^d$, $\cP^{\rm abs}_\rho(\R^d)\times \cP_\rho(\R^d)\ni(\nu,\mu)\mapsto \cJ_\rho(\nu,\mu)$ is continuous for $\cW_\rho$.
 \end{proposition}
 \begin{proof}[Proof of Proposition \ref{propcont}]
   We only prove the continuity of $\cP_\rho(\R^d)\times \cP_\rho(\R^d)\ni(\mu,\nu)\mapsto \cI_\rho(\mu,\nu)$ since, replacing \cite[Theorem 2.1]{AlCoJo20} by \cite[Theorem 4.1]{AlCoJo20} for the uniqueness of the projection, the proof is similar for the continuity properties of $\cJ_\rho$. Let $(\mu_n)_{n\ge 1}$ and $(\nu_n)_{n\ge 1}$ be sequences of elements of $\cP_\rho(\R^d)$ such that $\lim_{n\to\infty}\cW_\rho(\mu_n,\mu)=\lim_{n\to\infty}\cW_\rho(\nu_n,\nu)=0$ for some $\mu,\nu\in\cP_\rho(\R^d)$. With \eqref{lipdisti}, we deduce that $$\cW_\rho(\mu,\cI_\rho(\mu,\nu))=\lim_{n\to\infty}\cW_\rho(\mu_n,\cI_\rho(\mu_n,\nu_n))=\lim_{n\to\infty}\cW_\rho(\mu,\cI_\rho(\mu_n,\nu_n)).$$ The weakly converging sequence $(\mu_n)_{n\ge 1}$ is tight. The boundedness
   \begin{equation}\label{eq_uiI}
    \int_{\R^d}|y|^\rho\cI_\rho(\mu_{n_k},\nu_{n_k})(dy)=\cW_\rho^\rho(\delta_0,\cI_\rho(\mu_{n_k},\nu_{n_k}))\le \left(\cW_\rho(\delta_0,\mu)+\cW_\rho(\mu,\cI_\rho(\mu_{n_k},\nu_{n_k}))\right)^\rho
   \end{equation}
   combined with Markov inequality implies the tightness of $(\cI_\rho(\mu_n,\nu_n))_{n\ge 1}$. Therefore,  $(\pi_n)_{n\ge 1}$ is also tight, where, for $n\ge 1$, $\pi_n\in\Pi(\mu_n,\cI_\rho(\mu_n,\nu_n))$ is such that $\cW_\rho^\rho(\mu_n,\cI_\rho(\mu_n,\nu_n))=\int_{\R^d\times\R^d} |y-x|^\rho\pi_n(dx,dy)$. Let us consider a subsequence % that we still index by $n$ for notational simplicity
   $(\pi_{n_k})_{k\ge 1}$  converging weakly to $\pi_\infty$ and denote by $\cI$ the second marginal of $\pi_\infty$. We are going to check that $\cI=\cI_\rho(\mu,\nu)$ and $\lim_{k\to\infty}\cW_\rho(\cI_\rho(\mu_{n_k},\nu_{n_k}),\cI_\rho(\mu,\nu))=0$. Since from any subsequence of $(\pi_n)_{n\ge 1}$, we can extract a further subsequence which converges weakly, this implies the desired conclusion : $\lim_{n\to\infty}\cW_\rho(\cI(\mu_{n},\nu_{n}),\cI_\rho(\mu,\nu))=0$.

   The boundedness of~\eqref{eq_uiI}  together with the following consequence of H\"older's inequality  $$\int_A|y|\cI_\rho(\mu_{n_k},\nu_{n_k})(dy)\le \left(\cI_\rho(\mu_{n_k},\nu_{n_k})(A)\right)^{\frac{\rho-1}{\rho}}\left(\int_{\R^d}|y|^\rho\cI_\rho(\mu_{n_k},\nu_{n_k})(dy)\right)^{\frac 1\rho}$$ ensure that
 $$\lim_{\varepsilon\to 0+}\sup_{k\ge 1}\sup_{A\in\cB(\R^d):\cI_\rho(\mu_{n_k},\nu_{n_k})(A)\le \varepsilon}\int_A|y|\cI_\rho(\mu_{n_k},\nu_{n_k})(dy)=0.$$
With the weak convergence of $(\cI_\rho(\mu_{n_k},\nu_{n_k}))_{k\ge 1}$ to $\cI$ and \cite[Lemma 3.1 (d)]{BJMP22}, we deduce that $\lim_{k\to\infty}\cW_1(\cI_\rho(\mu_{n_k},\nu_{n_k}),\cI)=0$. By \cite[Theorem 6.9]{Vill09}, this implies that for $\varphi:\R^d\to\R$ continuous with at most affine growth, $\lim_{n\to\infty}\int_{\R^d}\varphi(x) \cI_\rho(\mu_{n_k},\nu_{n_k})(dx)=   \int_{\R^d}\varphi(x) \cI(dx)$ . On the other hand, $\lim_{k\to\infty}\int_{\R^d}\varphi(x)\nu_{n_k}(dx)=\int_{\R^d}\varphi(x)\nu(dx)$ since $\cW_1(\nu_{n_k},\nu)\le \cW_\rho(\nu_{n_k},\nu)\to 0$. When $\varphi$ is moreover convex, taking the limit $k\to\infty$ in the inequality $\int_{\R^d}    \varphi(x) \cI_\rho(\mu_{n_k},\nu_{n_k})(dx)\le \int_{\R^d}\varphi(x) \nu_{n_k}(dx)$ deduced from $\cI_\rho(\mu_{n_k},\nu_{n_k})\le_{cx}\nu_{n_k}$, we obtain $\int_{\R^d}    \varphi(x) \cI(dx)\le \int_{\R^d}\varphi(x) \nu(dx)$. By \cite[Lemma A.1]{AlCoJo20}, this ensures that $\cI\le_{cx}\nu$.

   Using that $\pi_\infty\in \Pi(\mu,\cI)$ then the continuity of the non-negative function $\R^d\times\R^d\ni(x,y)\mapsto|y-x|^\rho$, we get that \begin{align*}
                                                                                                                                                                                                                                                  \cW_\rho^\rho(\mu,\cI)&\le \int_{\R^d\times\R^d} |y-x|^\rho\pi_\infty(dx,dy)\le \liminf_{k\to\infty}\int_{\R^d\times\R^d} |y-x|^\rho\pi_{n_k}(dx,dy)\\&=\lim_{n\to\infty}\cW^\rho_\rho(\mu_{n},\cI_\rho(\mu_{n},\nu_{n}))=\cW^\rho_\rho(\mu,\cI_\rho(\mu,\nu)).\end{align*}
      With the definition of $\cI_\rho(\mu,\nu)$ and the uniqueness of this projection given in \cite[Theorem 2.1]{AlCoJo20}, we deduce that $\cI=\cI_\rho(\mu,\nu)$, $\pi_\infty$ is an optimal $\cW_\rho$ coupling between $\mu$ and $\cI_\rho(\mu,\nu)$ and $(\cI_\rho(\mu_{n_k},\nu_{n_k}))_{k\ge 1}$ converges weakly to $\cI_\rho(\mu,\nu)$ as $k\to\infty$. We are now going to check the stronger convergence $\lim_{k\to\infty}\cW_\rho(\cI_\rho(\mu_{n_k},\nu_{n_k}),\cI_\rho(\mu,\nu))=0$.  Let $\eta_n$ (resp. $\eta_\infty$) % with cumulative distribution function $F_n$ (resp. $F_\infty$)
      denote the image of $\pi_n$ (resp. $\pi_\infty$) by $\R^d\times\R^d\ni(x,y)\mapsto y-x$. By continuity of this function, $(\eta_{n_k})_{k\ge 1}$ converges weakly to $\eta_\infty$ as $k\to\infty$. Since $\int_\R |x|^\rho\eta_n(dx)=\cW_\rho^\rho(\mu_{n},\cI_\rho(\mu_{n},\nu_{n}))\underset{n\to\infty}{\longrightarrow} \cW_\rho^\rho(\mu,\cI_\rho(\mu,\nu))=\int_\R |x|^\rho\eta_\infty(dx)$, with \cite[Theorem 6.9]{Vill09}, we deduce that $\lim_{k\to\infty}\cW_\rho(\eta_{n_k},\eta_\infty)=0$.
      For $\varepsilon\in(0,1]$, we have
 \begin{align*}
   &\sup_{A\in\cB(\R^d):\cI_\rho(\mu_{n_k},\nu_{n_k})(A)\le \varepsilon}\int_A |y|^\rho\cI_\rho(\mu_{n_k},\nu_{n_k})(dy) \le \sup_{B\in\cB(\R^d\times\R^d):\pi_{n_k}(B)\le \varepsilon}\int_B|y|^\rho\pi_{n_k}(dx,dy)\\&\phantom{sup}\le \sup_{B\in\cB(\R^d\times\R^d):\pi_{n_k}(B)\le \varepsilon}\int_B2^{\rho-1}(|x|^\rho+|y-x|^\rho)\pi_{n_k}(dx,dy)\\&\phantom{sup}\le 2^{\rho -1}\sup_{A\in\cB(\R^d):\mu_{n_k}(A)\le \varepsilon}\int_{A}|x|^\rho\mu_{n_k}(dx)+2^{\rho -1}\sup_{A\in\cB(\R^d):\eta_{n_k}(A)\le \varepsilon}\int_{A}|x|^\rho\eta_{n_k}(dx)
 \end{align*}
where the first inequality follows from $\pi_{n_k}(\R^d\times A)=\cI_\rho(\mu_{n_k},\nu_{n_k})(A)$ and the last one from \cite[Lemma 3.1 (a)]{BJMP22} combined with  the fact that for $B\in\cB(\R^d\times\R^d)$, the images of $1_B(x,y)\pi_{n_k}(dx,dy)$ by $(x,y)\mapsto x$ and $(x,y)\mapsto y-x$ are respectively smaller than $\mu_{n_k}$ and $\eta_{n_k}$ . Then $\lim_{k\to\infty}\cW_\rho(\mu_{n_k},\mu)+\cW_\rho(\eta_{n_k},\eta_\infty)=0$ combined with 
\cite[Lemma 3.1 (d)]{BJMP22} ensures that the supremum over $k\ge 1$ of the right-hand side converges to $0$ with $\varepsilon$ and so does the supremum of the left-hand side.  With the boundedness of  $\left(\int_{\R^d}|y|^\rho\cI_\rho(\mu_{n_k},\nu_{n_k})(dy)\right)_{k\ge 1}$ and \cite[Lemma 3.1 (d)]{BJMP22}, we conclude that $\lim_{k\to\infty}\cW_\rho(\cI_\rho(\mu_{n_k},\nu_{n_k}),\cI_\rho(\mu,\nu))=0$.                                                    \end{proof}
\begin{remark} An easy adaptation of the end of the proof shows that $\lim_{n\to\infty}\cW_\rho(\gamma_n,\gamma)=0$ when  $\rho\ge 1$, $\mu,\gamma\in\cP_\rho(\R^d)$ and $(\mu_n)_{n\ge 1}$, $(\gamma_n)_{n\ge 1}$ are sequences of $\cP_\rho(\R^d)$ such that $\lim_{n\to\infty}\cW_\rho(\mu_n,\mu)=0$, $(\gamma_n)_{n\ge 1}$ converges weakly to $\gamma$ and $\lim_{n\to\infty}\cW_\rho(\mu_n,\gamma_n)=\cW_\rho(\mu,\gamma)$. 
\end{remark}\section{Regularity of $\cI_2(\mu,\nu)$ : non expansion in $\mu$ and H\"older continuity with exponent $\frac 12$ in $\nu$}
We denote by $m(\eta)=\int x\eta(dx)$ the barycenter of a probability measure $\eta\in\cP_1(\R^d)$. For $z\in\R^d$, let $\R^d\ni x\mapsto \tau_z(x)=x+z\in\R^d$ denote the translation by the vector $z$.
\begin{proposition}\label{propmoyproj}
  For $\mu,\nu\in{\cP}_2(\R^d)$, we have
  $$\mathcal{I}_2(\mu,\nu)=\mathcal{I}_2(\tau_{m(\nu)-m(\mu)}\#\mu,\nu)\mbox{ and }\mathcal{J}_2(\nu,\mu)=\mathcal{J}_2(\tau_{m(\mu)-m(\nu)}\#\nu,\mu).$$
\end{proposition}
\begin{proof}Let $\sigma,\theta\in\cP_2(\R^d)$ and $(\tau_{m(\theta)-m(\sigma)},I_d)(x,y)=(x+m(\theta)-m(\sigma),y)$ for $(x,y)\in\R^d\times\R^d$. By bias variance decomposition and since $\Pi(\sigma,\theta)\ni \pi\mapsto (\tau_{m(\theta)-m(\sigma)},I_d)\#\pi\in\Pi(\tau_{m(\theta)-m(\sigma)}\#\sigma,\theta)$ is a bijection, we have
\begin{align*}
   \cW^2_2(\sigma,\theta)&=\inf_{\pi\in\Pi(\sigma,\theta)}\int |x-y|^2\pi(dx,dy)\\&=\left|m(\sigma)-m(\theta)\right|^2+\inf_{\pi\in\Pi(\sigma,\theta)}\int\left|x+m(\theta)-m(\sigma)-y\right|^2\pi(dx,dy)\\&=\left|m(\sigma)-m(\theta)\right|^2+\cW^2_2(\tau_{m(\theta)-m(\sigma)}\#\sigma,\theta).
\end{align*}
Applying this equality with $(\sigma,\theta)$ equal to $(\mu,\eta)$ with $\eta\le_{cx}\nu$ so that $m(\eta)=m(\nu)$ (resp. $(\nu,\eta)$ with $\mu\le_{cx}\eta$ so that $m(\eta)=m(\mu)$), we deduce that $\mathcal{I}_2(\mu,\nu)=\mathcal{I}_2(\tau_{m(\nu)-m(\mu)}\#\mu,\nu)$ (resp. $\mathcal{J}_2(\nu,\mu)=\mathcal{J}_2(\tau_{m(\mu)-m(\nu)}\#\nu,\mu)$).
\end{proof}
In view of this lemma, it is natural to introduce the centered Wasserstein distance between two elements $\mu$ and $\nu$ of $\cP_2(\R^d)$ :
\begin{align*}
   \bW(\mu,\nu)&=\left(\cW^2_{2}(\mu,\nu)-\left|m(\mu)-m(\nu)\right|^2\right)^{1/2}\\&=\cW_2(\tau_{m(\nu)-m(\mu)}\#\mu,\nu)=\cW_2(\tau_{m(\mu)-m(\nu)}\#\nu,\mu).
\end{align*}
Strictly speaking, $\bW(\mu,\nu)$ is not a distance on $\cP_2(\R^d)$, but this is a distance on each level set $\{ \mu \in \cP_2(\R^d) : m(\mu)=c \}$, $c \in \R^d$ (or equivalently on the quotient set $\cP_2(\R^d)/\sim$ with $\mu\sim\nu \iff m(\mu)=m(\nu)$).  

\begin{proposition}\label{lipmu}
   For $\mu,\tilde\mu,\nu\in \mathcal{P}_2(\R^d)$, $\cW_2(\mathcal{I}_2(\mu,\nu)
,\mathcal{I}_2(\tmu,\nu))\le \bW(\mu,\tmu)$. 
\end{proposition}
\noindent Let us note that since $m(\mathcal{I}_2(\mu,\nu))=m(\nu)=m(\mathcal{I}_2(\tmu,\nu))$, $\cW_{2}(\mathcal{I}_2(\mu,\nu),\mathcal{I}_2(\tmu,\nu))=\bW(\mathcal{I}_2(\mu,\nu),\mathcal{I}_2(\tmu,\nu))$.

When $d=1$, by Theorem 1.1 \cite{JMP23}, for $\mu,\tilde\mu,\nu,\tnu\in \mathcal{P}_2(\R)$, $\cW_2(\mathcal{I}(\tmu,\nu)
,\mathcal{I}(\tmu,\tnu))\le \cW_2(\nu,\tnu)$. With the triangle inequality and the estimation in Proposition \ref{lipmu}, we deduce the following improvement of the estimation in this Theorem without the additional factor $2$ for the first term in the right-hand side. Note that according to Example 1.3 \cite{JMP23}, this factor is optimal for the Wasserstein distance $\cW_1$.\begin{corollary}
For $\mu,\tilde\mu,\nu,\tnu\in \mathcal{P}_2(\R)$, $\cW_2(\mathcal{I}(\mu,\nu)
,\mathcal{I}(\tmu,\tnu))\le \bW(\mu,\tmu)+\cW_2(\nu,\tnu)$.
\end{corollary}
By Proposition 3.1 \cite{AlCoJo20}, for $\mu,\tilde\mu,\nu,\tnu\in \mathcal{P}_2(\R^d)$ such that $\mu\le_{cx}\nu$,
$\cW_2(\mu,\cI_2(\mu,\tnu))\le \cW_2(\nu,\tnu)$. With the triangle inequality and the estimation in Proposition \ref{lipmu} with $\nu$ replaced by $\tnu$, we deduce the following improvement of the estimation \eqref{lipproi} in this Proposition without the additional factor $2$ for the first term in the right-hand side.
\begin{corollary}
For $\mu,\tilde\mu,\nu,\tnu\in \mathcal{P}_2(\R^d)$ such that $\mu\le_{cx}\nu$, $\cW_2(\mu,\mathcal{I}_2(\tmu,\tnu))\le \bW(\mu,\tmu)+\cW_2(\nu,\tnu)$.
\end{corollary}
\begin{remark}
  Minimizing  the right hand side w.r.t. $\nu$ such that $\mu \le_{cx} \nu$ gives
   $$\forall \mu,\tilde\mu,\tnu\in \mathcal{P}_2(\R^d),\;\cW_2(\mu,\mathcal{I}_2(\tmu,\tnu))\le \bW(\mu,\tmu)+\cW_2(\mathcal{J}_2(\tnu,\mu),\tnu).$$
\end{remark}

To check Proposition \ref{lipmu}, we generalize the argument proving that the projection on a closed convex subset of a Hilbert space is $1$-Lipschitz. Note that while $\{\eta\in\cP_2(\R^d):\eta\le_{cx}\nu\}$ is a closed convex subset of $\cP_2(\R^d)$ endowed with the Wasserstein distance $\cW_2$, $\cP_2(\R^d)$ is not a vector space.

We first recall~\cite[Theorem 2.1]{AlCoJo17} that will be often used in this section. 
\begin{theorem}\label{thm_ACJ}
Let $\rho\ge 1$ and $\mu,\nu \in \mathcal{P}_\rho(\R^d)$. Then, we have $$W_\rho^\rho(\mu,\cI_\rho(\mu,\nu))=\min_{\pi \in \Pi(\mu,\nu)} \int_{\R^d} \left| x - \int_{\R^d} y \, \pi_x(dy) \right|^\rho \, \mu(dx).$$ 
For $\rho\ge 1$, there exists a (unique when $\rho>1$) minimizer $\pi^\star \in \Pi(\mu,\nu)$ of the right-hand side, and it satisfies $\cI_\rho(\mu,\nu)=m(\pi^\star_\cdot) \# \mu$.
\end{theorem}

\begin{lemma}\label{lemma_wwp}Let $\rho\ge 1$, $\mu,\nu,\tmu,\tnu\in\cP_\rho(\R^d)$ and $k\in\Pi(\mu,\nu)$, $\tilde{k}\in \Pi(\tmu,\tnu)$ be such that $\mathcal{I}_\rho (\mu,\nu)=m(k_\cdot)\#\mu$, $\mathcal{I}_\rho(\tmu,\tnu)=m(\tk_\cdot)\#\tmu$. Then
 \begin{equation}
   \cW_\rho(\cI_\rho(\mu,\nu),\cI_\rho(\tmu,\tnu))=\inf_{\pi\in\Pi(\mu,\tmu)}\left(\int |m(k_x)-m({\tk}_{\tx})|^\rho\pi(dx,d\tx)\right)^{1/\rho}.\label{wwp}
 \end{equation}
\end{lemma}
\begin{proof}
  The existence of $k$ and $\tilde{k}$ is given by Theorem~\ref{thm_ACJ}. Since for $\pi \in \Pi(\mu,\tmu)$, the image of $\pi$ by $\R^{d}\times\R^d\ni(x,\tx)\mapsto (m(k_x),m(\tk_{\tx}))$ belong to $\Pi(\cI_\rho(\mu,\nu),\cI_\rho(\tmu,\tnu))$, the left-hand side in \eqref{wwp} is not greater the the right-hand side. To prove the converse inequality, we choose $\hat\pi_\star\in\Pi(\cI_\rho(\mu,\nu),\cI_\rho(\tmu,\tnu))$ as an optimal $\cW_\rho$ coupling between $\cI_\rho(\mu,\nu)$ and $\cI_\rho(\tmu,\tnu)$. We have $\mu(dx)\delta_{m(k_x)}(dy)=\cI_\rho(\mu,\nu)(dy)\chi_y(dx)$ and $\tmu(d\tx)\delta_{m(\tk_{\tx})}(d\ty)=\cI_\rho(\tmu,\tnu)(d\ty)\tchi_{\ty}(d\tx)$ for some Markov kernels $\chi_\cdot(\cdot)$ and $\tchi_\cdot(\cdot)$. Then $\pi_\star(dx,d\tx):=\int_{(y,\ty)\in\R^d\times\R^d}\chi_y(dx)\tchi_{\ty}(d\tx)\hat\pi_\star(dy,d\ty)\in\Pi(\mu,\tmu)$ and since $$\int_{(\tx,\ty)\in\R^d\times\R^d}\chi_y(dx)\tchi_{\ty}(d\tx)\hat\pi_\star(dy,d\ty)=\mu(dx)\delta_{m(k_x)}(dy),$$ $\chi_y(dx)\tchi_{\ty}(d\tx)\hat\pi_\star(dy,d\ty)$ a.e., $m(k_x)=y$ and, in a symmetric way, $m(\tk_{\tx})=\ty$. Therefore
  \begin{align*}
   \int_{(x,\tx)\in\R^{2d}}|m(k_x)-m(\tk_{\tx})|^\rho\pi_\star(dx,d\tx)&=\int_{(x,\tx,y,\ty)\in\R^{4d}}|y-\ty|^\rho\chi_y(dx)\tchi_{\ty}(d\tx)\hat\pi_\star(dy,d\ty)\\&=\cW_\rho^\rho(\cI_\rho(\mu,\nu),\cI_\rho(\tmu,\tnu)).
  \end{align*}
\end{proof}

\begin{proof}[Proof of Proposition \ref{lipmu}]
Let $k(dx,dy)=\mu(dx)k_x(dy)\in\Pi(\mu,\nu)$ (resp. $\tilde{k}(d\tx,dy)=\tmu(d\tx)\tk_{\tx}(dy)\in \Pi(\tmu,\nu)$) such that $\mathcal{I}_2(\mu,\nu)=m(k_\cdot)\#\mu$ (resp. $\mathcal{I}_2(\tmu,\nu)=m({\tilde k}_\cdot)\#\tmu$) by given by Theorem~\ref{thm_ACJ}. The kernel $k$ minimizes $\int |x-m(\pi_x)|^2\mu(dx)$ over $\pi\in\Pi(\mu,\nu)$ and $\cW_2^2(\mu,\cI(\mu,\nu))=\int|x-m(k_x)|^2\mu(dx)$. 
  For $\pi\in\Pi(\mu,\nu)$, we have $\varepsilon \pi +(1-\varepsilon)k\in\Pi(\mu,\nu)$ and $m((\varepsilon \pi +(1-\varepsilon)k)_x)=\varepsilon m(\pi_x)+(1-\varepsilon)m(k_x)$ for each $\varepsilon\in[0,1]$ and $x\in\R$. Moreover,
  \begin{align*}
   |x-m(k_x)|^2-|x-m((\varepsilon \pi +(1-\varepsilon)k)_x)|^2=2\varepsilon\langle m(k_x)-m(\pi_x),m(k_x)-x\rangle-\varepsilon^2|m(k_x)-m(\pi_x)|^2.
  \end{align*}
Therefore, the inequality $\int |x-m(k_x)|^2-|x-m((\varepsilon \pi +(1-\varepsilon)k)_x)|^2\mu(dx)\le 0$ implies in the limit $\varepsilon\to 0$ that
\begin{equation}
   \forall \pi\in\Pi(\mu,\nu),\;\int \langle m(k_x)-m(\pi_x),m(k_x)-x\rangle \mu(dx)\le 0.\label{optik}
\end{equation}
In the same way,
\begin{equation}
   \forall \tpi\in\Pi(\tmu,\nu),\;\int \langle m(\tk_{\tx})-m(\tpi_{\tx}),m(\tk_{\tx})-\tx\rangle \tmu(d\tx)\le 0.\label{optitk}
\end{equation}
 Let now $\chi\in\Pi(\mu,\tmu)$ and $\tchi$ denotes its image by $(x,\tx)\mapsto (\tx,x)$. We have
 \begin{align*}
  \int |m(k_x)-m(\tk_{\tx})|^2\chi(dx,d\tx)=&\int \langle m(k_x)-m(\tk_{\tx}),m(k_x)-x+\tx-m(\tk_{\tx})+x-\tx \rangle \chi(dx,d\tx)\\
  =&
  \int\langle m(k_x)- m({\chi\tk}_x), m(k_x)-x\rangle\mu(dx)\\&+\int\langle m(\tk_{\tx})- m(\tchi k_{\tx}), m(\tk_{\tx})-\tx\rangle\tmu(d\tx)\\&+\int\langle m(k_x)-m(\tk_{\tx}), x-\tx\rangle\chi(dx,d\tx).
\end{align*}
Since $\chi\tk\in\Pi(\mu,\nu)$ and $\tchi k\in\Pi(\tmu,\nu)$, the first two terms in the right-hand side are non-positive according to \eqref{optik} and \eqref{optitk} respectively.
Applying Cauchy-Schwarz inequality to the third term, and choosing $\chi$ as an optimal $\cW_2$ coupling $\chi_\star$ between $\mu$ and $\tmu$, we deduce that
$$\int |m(k_x)-m(\tk_{\tx})|^2\chi_\star(dx,d\tx)\le \left(\int |m(k_x)-m(\tk_{\tx})|^2\chi_\star(dx,d\tx)\right)^{1/2} \cW_2(\mu,\tmu),$$
so that $\int |m(k_x)-m(\tk_{\tx})|^2\chi_\star(dx,d\tx)\le \cW^2_2(\mu,\tmu)$.
With Lemma~\ref{lemma_wwp}, we deduce that $$\cW_2(\cI_2(\mu,\nu),\cI_2(\tmu,\nu))\le\cW_2(\mu,\tmu).$$ Since, by Proposition \ref{propmoyproj}, $\cI_2(\mu,\nu)=\cI_2(\tau_{m(\nu)-m(\mu)}\#\mu,\nu)$ and $\cI_2(\tmu,\nu)=\cI_2(\tau_{m(\nu)-m(\tmu)}\#\tmu,\nu)$, replacing $(\mu,\tmu)$ by $(\tau_{m(\nu)-m(\mu)}\#\mu,\tau_{m(\nu)-m(\tmu)}\#\tmu)$, we conclude that $$\cW_2(\cI_2(\mu,\nu),\cI_2(\tmu,\nu))\le\cW_2(\tau_{m(\nu)-m(\mu)}\#\mu,\tau_{m(\nu)-m(\tmu)}\#\tmu)=\bW(\mu,\tmu).$$
  \end{proof}

  \begin{remark}\label{rk_lift}
    An alternative proof for Proposition~\ref{lipmu} would be to lift on an atomless probability space $(\Omega,\mathcal{F},\PP)$ and to define, for $X,Y\in L^2(\Omega,\R^d)$, ${\rm I}_2(X,Y)$ as the $L^2$ projection of~$X$ on the closed convex set $\{Z \in L^2(\Omega,\R^d): {\mathcal L}(Z)\le_{cx} {\mathcal L}(Y) \}$. One easily checks that for $X\sim\mu\in\cP_2(\R^d)$ and $Y\sim\nu\in\cP_2(\R^d)$,  $\mathcal{L}({\rm I}_2(X,Y))=\cI_2(\mu,\nu)$, since $\cI_2(\mu,\nu)$ is the unique minimizer of $\cW_2(\mu,\eta)$, $\eta \le_{cx} \nu$. Then, the one-Lipschitz property of $L^2$-projections on closed convex sets gives 
    $$\E[|{\rm I}_2(X,Y)-{\rm I}_2(X',Y)|^2]^{1/2}\le  \E[|X-X'|^2]^{1/2}.$$
    We also observe that ${\rm I}_2(X+c,Y)={\rm I}_2(X,Y)$ for any $c\in \R^d$, so that  ${\rm I}_2(X,Y)={\rm I}_2(X-\E[X],Y)$. Combined with the previous inequality, we get precisely Proposition~\ref{lipmu}. 

    It is interesting to note that $\{Z \in L^2(\Omega,\R^d): {\mathcal L}(Y)\le_{cx} {\mathcal L}(Z) \}$ may not be convex, so that there may be different random variables minimizing $\E[|X-Z|^2]$ in this closed set. Thus, it is not clear how to get an analogous result for $\cJ_2(\nu,\mu)$.  
  \end{remark}

\begin{proposition}\label{holdnu}
   For $\mu,\nu,\tnu\in \mathcal{P}_2(\R^d)$, \begin{equation}
      \bW^2(\mathcal{I}_2(\mu,\nu)
,\mathcal{I}_2(\mu,\tnu))\le \left(\bW(\mu,\mathcal{I}_2(\mu,\nu))+\bW(\mu,\mathcal{I}_2(\mu,\tnu))\right)\bW(\nu,\tnu).\label{regnu}
   \end{equation}
 \end{proposition}
\begin{remark}
  Let $\mu,\nu,\tnu\in \mathcal{P}_2(\R^d)$ be such that $\mu\le_{cx}\nu$. Then $\mathcal{I}_2(\mu,\nu)=\mu$ so that the estimation \eqref{regnu} implies that $\bW(\mu,\mathcal{I}_2(\mu,\tnu))\le \bW(\nu,\tnu)$. Since $\mu$ and $\nu$ (resp. $\mathcal{I}_2(\mu,\tnu)$ and $\tnu$) share the same expectation, we deduce that $\cW_{2}(\mu,\mathcal{I}_2(\mu,\tnu))\le \cW_{2}(\nu,\tnu)$ which also is a consequence of Proposition 3.1 \cite{AlCoJo20}. 
\end{remark}
\begin{proof}
Let $k(dx,dy)=\mu(dx)k_x(dy)\in\Pi(\mu,\nu)$ (resp. $\tilde{k}(dx,d\ty)=\mu(dx)\tk_{x}(d\ty)\in \Pi(\mu,\tnu)$) such that $\mathcal{I}_2(\mu,\nu)=m(k_\cdot)\#\mu$ (resp. $\mathcal{I}_2(\mu,\tnu)=m({\tilde k}_\cdot)\#\mu$) be given by Theorem~\ref{thm_ACJ}.
   Let also $\sigma\in\Pi(\nu,\tnu)$ by optimal for $\cW_2(\nu,\tnu)$ and $\tilde{\sigma}$ denote its image by $(y,\ty)\mapsto (\ty,y)$. Using $\int m(k_x)\mu(dx)=m(\nu)$ and $\int m(\tk_x)\mu(dx)=m(\tnu)$, we have 
   \begingroup
  \allowdisplaybreaks 
   \begin{align}
     \int|m(k_x)-m(\nu)&-m(\tk_{x})+m(\tnu)|^2\mu(dx)\notag\\=& \int\langle m(k_x)-m(\tk_{x}),m(k_x)-m(\nu)-m(\tk_{x})+m(\tnu)\rangle\mu(dx)\notag\\
     =& \int\langle m(k_x)-m(\tk_{x}),m(k_x)-m(\nu)-x+m(\mu)\rangle\mu(dx)\notag\\&+ \int\langle m(k_x)-m(\tk_{x}),x-m(\mu)-m(\tk_{x})+m(\tnu)\rangle\mu(dx)\notag\\
     =&\int\langle m(k_x)-m({\tk\tilde\sigma}_x),m(k_x)-m(\nu)-x+m(\mu)\rangle\mu(dx)\notag\\
     &+\int\langle m({\tk\tilde\sigma}_x)-m({\tk}_x),m(k_x)-m(\nu)-x+m(\mu)\rangle\mu(dx)\notag\\
     &
     +\int\langle m(k_x)-m({k\sigma}_x),x-m(\mu)-m({\tk}_x)+m(\tnu)\rangle\mu(dx)\notag\\
     &
       +\int\langle m({k\sigma}_x)-m({\tk}_x),x-m(\mu)-m({\tk}_x)+m(\tnu)\rangle\mu(dx).\label{majoprojnu}
\end{align}
\endgroup
Since $\int m(k_x)\mu(dx)=m(\nu)=\int m({\tk\tilde\sigma}_x) \mu(dx)$, the first term in the right-hand side is equal to $\int\langle m(k_x)-m({\tk\tilde\sigma}_x),m(k_x)-x\rangle\mu(dx)$ and is non-positive by \eqref{optik} applied with $\pi=\tk\tilde\sigma$. In the same way, the fourth term in the right-hand side is non-positive. Since $\int (m(k_x)-m(\nu)-x+m(\mu))\mu(dx)=0$, the second term is equal to
$\int\langle m({\tk\tilde\sigma}_x)-m(\nu)-m({\tk}_x)+m(\tnu),m(k_x)-m(\nu)-x+m(\mu)\rangle\mu(dx)$ and, by Cauchy-Schwarz inequality, it is bounded from above by $$\left(\int|m({\tk\tilde\sigma}_x)-m(\nu)-m({\tk}_x)+m(\tnu)|^2\mu(dx)\right)^{1/2}\bW(\mu,\mathcal{I}_2(\mu,\nu)).$$
Using that $m({\tk\tsigma}_x)=\int_{\ty\in\R^d}m(\tsigma_y)\tk_x(d\ty)$ and applying twice Jensen's inequality, we get
\begingroup
  \allowdisplaybreaks 
     \begin{align*}
   \int|m({\tk\tsigma}_x)&-m(\nu)-m({\tk}_x)+m(\tnu)|^2\mu(dx)\\=&\int_{x\in\R^d}\left|\int_{\ty\in\R^d} (m(\tsigma_{\ty})-m(\nu)-\ty+m(\tnu))\tk_x(d\ty)\right|^2\mu(dx)\\&\le \int_{(x,\ty)\in\R^{2d}}   \left|m(\tsigma_{\ty})-m(\nu)-\ty+m(\tnu)\right|^2\tk(dx,d\ty)\\
     &=\int_{\ty\in\R^{d}}\left|\int_{y\in\R^d}(y-m(\nu)-\ty+m(\tnu))\tsigma_{\ty}(dy)\right|^2\tnu(d\ty)\\&\le \int_{(y,\ty)\in\R^{2d}}|y-m(\nu)-\ty+m(\tnu)|^2\tsigma(d\ty,dy)=\bW^2(\nu,\tnu).\end{align*}
     \endgroup
   Therefore $\int\langle m({\tk\tilde\sigma}_x)-m({\tk}_x),m(k_x)-m(\nu)-x+m(\mu)\rangle\mu(dx)\le\bW(\nu,\tnu)\bW(\mu,\mathcal{I}_2(\mu,\nu))$. In the same way, $\int\langle m(k_x)-m({k\sigma}_x),x-m(\mu)-m({\tk}_x)+m(\tnu)\rangle\mu(dx) \le\bW(\nu,\tnu)\bW(\mu,\mathcal{I}_2(\mu,\nu))$.
   Plugging the upper-bounds for the four terms in the right-hand side of \eqref{majoprojnu} in this equation, we deduce that
$$\int|m(k_x)-m(\nu)-m(\tk_{x})+m(\tnu)|^2\mu(dx)\le \left(\bW(\mu,\mathcal{I}_2(\mu,\nu))+\bW(\mu,\mathcal{I}_2(\mu,\nu))\right)\bW(\nu,\tnu).$$

  We conclude with the estimation \begin{align*}
   \bW^2(\mathcal{I}_2(\mu,\nu)
                                    ,\mathcal{I}_2(\mu,\tnu))&=\cW^2(\mathcal{I}_2(\mu,\nu)
  ,\mathcal{I}_2(\mu,\tnu))-|m(\nu)-m(\tnu)|^2\\&\le \int|m(k_x)-m({\tk}_x)|^2\mu(dx)-|m(\nu)-m(\tnu)|^2\\&=\int|m(k_x)-m(\nu)-m({\tk}_x)+m(\tnu)|^2\mu(dx)
  \end{align*} deduced from \eqref{wwp} for the coupling $\pi(dx,d\tx)=\mu(dx)\delta_{x}(d\tx)\in\Pi(\mu,\mu)$.
\end{proof}

\section{Projections of Gaussian distributions}

Let $\mu,\nu\in\cP_2(\R^d)$ with expectations $m_\mu,m_\nu\in\R^d$ and covariance matrices $\Sigma_\mu,\Sigma_\nu\in\cS^+_d$. For any coupling $\pi$ between $\mu$ and $\nu$, the covariance matrix of $\pi$ writes $\left(\begin{array}{cc}\Sigma_\mu & \Theta_\pi\\\Theta_\pi^* & \Sigma_{\nu}\end{array}\right)\in\cS^+_{2d}$ for some $\Theta_\pi\in\R^{d\times d}$ and, by bias variance decomposition,
\begin{align}
  \int_{\R^d\times \R^d}&|x-y|^2\pi(dx,dy)=|m_\mu-m_\nu|^2+\int_{\R^d\times \R^d}|(x-m_\mu)-(y-m_\nu)|^2\pi(dx,dy)\notag\\&=|m_\mu-m_\nu|^2+\int_{\R^d\times \R^d}|x-m_\mu|^2+|y-m_\nu)|^2-2(x-m_\mu)^*(y-m_\nu)\pi(dx,dy)\notag\\&=|m_\mu-m_\nu|^2+\tr(\Sigma_\mu+\Sigma_\nu-2\Theta_\pi).\label{w22}\end{align}
As a consequence, minimizing $\int_{\R^d\times \R^d}|x-y|^2\pi(dx,dy)$ with respect to $\pi\in\Pi(\mu,\nu)$ amounts to maximizing $\tr(\Theta_\pi)$. % When $\mu=\cN_d(m_\mu,\Sigma_\mu)$ and $\nu=\cN_d(m_\nu,\Sigma_\nu)$, then 
For each $\Theta\in\cS^+_d$ such that $\Gamma_\Theta:=\left(\begin{array}{cc}\Sigma_\mu & \Theta\\\Theta^* & \Sigma_{\nu}\end{array}\right)\in\cS^+_{2d}$, $\cN_{2d}\left(\left(\begin{array}{c}m_\mu\\ m_\nu\end{array}\right),\Gamma_\Theta\right)$ is a coupling between $\cN_d(m_\mu,\Sigma_\mu)$ and $\cN_d(m_\nu,\Sigma_\nu)$ so that
\begin{align*}
  \cW_2^2\left(\cN_d(m_\mu,\Sigma_\mu),\cN_d(m_\nu,\Sigma_\nu)\right)&=|m_\mu-m_\nu|^2+\tr(\Sigma_\mu+\Sigma_\nu)-2\sup_{\Theta\in\R^{d\times d}:\Gamma_\Theta\in\cS^+_d}\tr(\Theta)
\end{align*}
This equality was exploited by \cite{Dowlan,GiSho,Olpuk} to derive
\begin{equation}
   \cW_2^2\left(\cN_d(m_\mu,\Sigma_\mu),\cN_d(m_\nu,\Sigma_\nu)\right)=|m_\mu-m_\nu|^2+\tr\left(\Sigma_\mu+\Sigma_\nu-2(\Sigma_\mu^{1/2}\Sigma_\nu\Sigma_\mu^{1/2})^{1/2}\right)\label{w2gauss}
\end{equation}
with optimal $\Theta$ equal to $\Sigma_\mu^{1/2}(\Sigma_\mu^{1/2}\Sigma_\nu\Sigma_\mu^{1/2})^{1/2}\Sigma_\mu^{-1/2}$ (resp. $\Sigma_\nu^{-1/2}(\Sigma_\nu^{1/2}\Sigma_\mu\Sigma_\nu^{1/2})^{1/2}\Sigma_\nu^{1/2}$) when $\Sigma_\mu$ is positive definite (resp. $\Sigma_\nu$ is positive definite). Of course $\tr\left((\Sigma_\mu^{1/2}\Sigma_\nu\Sigma_\mu^{1/2})^{1/2}\right)= \tr\left((\Sigma_\nu^{1/2}\Sigma_\mu\Sigma_\nu^{1/2})^{1/2}\right)$ for all $\Sigma_\mu,\Sigma_\nu\in\cS^+_d$.\\Since $\{\Theta_\pi\in\R^{d\times d}:\pi\mbox{ coupling between }\mu\mbox{ and }\nu\}\subset\{\Theta\in\R^{d\times d}:\Gamma_\Theta\in\cS^+_d\}$, one also has
\begin{equation}
   \cW_2^2(\mu,\nu)\ge |m_\mu-m_\nu|^2+\tr\left(\Sigma_\mu+\Sigma_\nu-2(\Sigma_\mu^{1/2}\Sigma_\nu\Sigma_\mu^{1/2})^{1/2}\right)\label{minow2gaus}.\end{equation}

\subsection{{Existence of Gaussian projections}}
\begin{proposition}\label{prop_projG}
  Let $m_\mu,m_\nu\in\R^d$, $\Sigma_\mu,\Sigma_\nu\in\cS^+_d$. Then $\cI_2(\cN_d(m_\mu,\Sigma_\mu),\cN_d(m_\nu,\Sigma_\nu))=\cN_d(m_\nu,\Sigma_{\cI})$ where $\Sigma_{\cI}\in\cS^+_d$ is the unique solution of $\inf_{\Sigma\in\cS^+_d:\Sigma\le\Sigma_\nu}{\rm tr}\left(\Sigma-2(\Sigma_\mu^{1/2}\Sigma\Sigma_\mu^{1/2})^{1/2}\right)$. Any $\cW_2$-projection of $\cN_d(m_\nu,\Sigma_\nu)$ on the set of probability measures dominating $\cN_d(m_\mu,\Sigma_\mu)$ has mean $m_\mu$ and a covariance matrix solving $\inf_{\Sigma\in\cS^+_d:\Sigma_\mu\le\Sigma}{\rm tr}\left(\Sigma-2(\Sigma_\nu^{1/2}\Sigma\Sigma_\nu^{1/2})^{1/2}\right)$, and for any solution $\Sigma_\cJ$, $\cN_d(m_\mu,\Sigma_\cJ)$ is such a projection. This is the unique projection when $d=1$ or $\Sigma_\nu$ is positive definite and uniqueness of the solutions $\Sigma_\cJ$ holds in these cases. Moreover, for all $\mu,\nu\in\cP_2(\R^d)$ with respective expectations $m_\mu,m_\nu$ and covariance matrices $\Sigma_\mu,\Sigma_\nu$,
\begin{align*}
   \cW_2^2(\mu,\cI_2(\mu,\nu))&\ge |m_\mu-m_\nu|^2+\inf_{\Sigma\in\cS^+_d:\Sigma\le\Sigma_\nu}{\rm tr}\left(\Sigma_\mu+\Sigma-2(\Sigma_\mu^{1/2}\Sigma\Sigma_\mu^{1/2})^{1/2}\right)
  \\\inf_{\eta\in\cP_2(\R^d):\mu\le_{cx}\eta}\cW^2_2(\nu,\eta)&\ge |m_\mu-m_\nu|^2+ \inf_{\Sigma\in\cS^+_d:\Sigma_\mu\le\Sigma}{\rm tr}\left(\Sigma_\nu+\Sigma-2(\Sigma_\nu^{1/2}\Sigma\Sigma_\nu^{1/2})^{1/2}\right).
\end{align*}
\end{proposition}                                                              \begin{proof}
 Let $\Gamma=\left(\begin{array}{cc}\Sigma_\mu & \Theta\\\Theta^* & \Sigma_{\cI}
                   \end{array}\right)$ denote the covariance matrix of the optimal coupling between $\cN_d(m_\mu,\Sigma_\mu)$ and  $\cI_2(\cN_d(m_\mu,\Sigma_\mu),\cN_d(m_\nu,\Sigma_\nu))$. The inequality $$\cI_2(\cN_d(m_\mu,\Sigma_\mu),\cN_d(m_\nu,\Sigma_\nu))\le_{cx}\cN_d(m_\nu,\Sigma_\nu)$$ implies that $\Sigma_\cI\le\Sigma_\nu$ and therefore $\cN_d(m_\nu,\Sigma_\cI)\le_{cx}\cN_d(m_\nu,\Sigma_\nu)$ and that the expectation of $\cI_2(\cN_d(m_\mu,\Sigma_\mu),\cN_d(m_\nu,\Sigma_\nu))$ is $m_\nu$, so that, by \eqref{w22}, $$\cW_2^2(\cN_d(m_\mu,\Sigma_\mu),\cI_2(\cN_d(m_\mu,\Sigma_\mu),\cN_d(m_\nu,\Sigma_\nu)))=|m_\mu-m_\nu|^2+{\rm tr}\left(\Sigma_\mu+\Sigma_{\cI}-2\Theta\right).$$
                 Since $\cN_{2d}\left(\left(\begin{array}{c}m_\mu\\ m_\nu\end{array}\right),\Gamma\right)$ is a coupling between $\cN_d(m_\mu,\Sigma_\mu)$ and $\cN_d(m_\nu,\Sigma_\cI)$, with \eqref{w22} again, we deduce that
                   $$\cW_2^2(\cN_d(m_\mu,\Sigma_\mu),\cI_2(\cN_d(m_\mu,\Sigma_\mu),\cN_d(m_\nu,\Sigma_\nu)))\ge \cW_2^2(\cN_d(m_\mu,\Sigma_\mu),\cN_d(m_\nu,\Sigma_\cI))$$
                   and therefore $\cI_2(\cN_d(m_\mu,\Sigma_\mu),\cN_d(m_\nu,\Sigma_\nu))=\cN_d(m_\nu,\Sigma_\cI)$. As a consequence, $$\cW_2^2(\cN_d(m_\mu,\Sigma_\mu),\cN_d(m_\nu,\Sigma_\cI))=\inf_{\Sigma\in\cS^+_d:\Sigma\le\Sigma_\nu}\cW_2^2(\cN_d(m_\mu,\Sigma_\mu),\cN_d(m_\nu,\Sigma))$$ and, by \eqref{w2gauss}, $\Sigma_\cI$ solves $\inf_{\Sigma\in\cS^+_d:\Sigma\le\Sigma_\nu}{\rm tr}\left(\Sigma-2(\Sigma_\mu^{1/2}\Sigma\Sigma_\mu^{1/2})^{1/2}\right)$. Since for any other solution $\Sigma$, $\cW_2^2(\cN_d(m_\mu,\Sigma_\mu),\cN_d(m_\nu,\Sigma))=\cW_2^2(\cN_d(m_\mu,\Sigma_\mu),\cN_d(m_\nu,\Sigma_\cI))$, the uniqueness of the projection $\cI_2(\cN_d(m_\mu,\Sigma_\mu),\cN_d(m_\nu,\Sigma_\nu))$ implies that there is a unique solution.
                   Last, since $\cI_2(\mu,\nu)\le_{cx}\nu$ implies that the covariance matrix $\Sigma_{\cI_2(\mu,\nu)}$ of $\cI_2(\mu,\nu)$ satisfies $\Sigma_{\cI_2(\mu,\nu)}\le \Sigma_\nu$, the inequality $$\cW_2(\mu,\cI_2(\mu,\nu))\ge |m_\mu-m_\nu|^2+\inf_{\Sigma\in\cS^+_d:\Sigma\le\Sigma_\nu}{\rm tr}\left(\Sigma_\mu+\Sigma-2(\Sigma_\mu^{1/2}\Sigma\Sigma_\mu^{1/2})^{1/2}\right) %\cW_2(\cN_d(m_\mu,\Sigma_\mu),\cI_2(\cN_d(m_\mu,\Sigma_\mu),\cN_d(m_\nu,\Sigma_\nu)))
                   $$ follows from \eqref{minow2gaus} applied with $\nu$ replaced by $\cI_2(\mu,\nu)$.
                   
      Let now $\Gamma=\left(\begin{array}{cc}\Sigma_\nu & \Theta\\\Theta^* & \Sigma_{\cJ}
                   \end{array}\right)$ denote the covariance matrix of an optimal coupling between $\cN_d(m_\nu,\Sigma_\nu)$ and some $\cW_2$-projection (unique when $d=1$ or $\Sigma_\nu\in\cS^{+,*}_d$) $\cJ$ of this measure on the set of probability measures dominating $\cN_d(m_\mu,\Sigma_\mu)$ in the convex order. The inequality $$\cN_d(m_\mu,\Sigma_\mu)\le_{cx} \cJ$$ implies that $\Sigma_\mu\le\Sigma_\cJ$ and therefore $\cN_d(m_\mu,\Sigma_\mu)\le_{cx}\cN_d(m_\mu,\Sigma_\cJ)$ and that the expectation of $\cJ$ is $m_\mu$, so that, by \eqref{w22}, $$\cW_2^2(\cN_d(m_\nu,\Sigma_\nu),\cJ)=|m_\mu-m_\nu|^2+{\rm tr}\left(\Sigma_\nu+\Sigma_{\cJ}-2\Theta\right).$$
                 Since $\cN_{2d}\left(\left(\begin{array}{c}m_\nu\\ m_\mu\end{array}\right),\Gamma\right)$ is a coupling between $\cN_d(m_\nu,\Sigma_\nu)$ and $\cN_d(m_\mu,\Sigma_\cJ)$, with \eqref{w22} again, we deduce that
                   $$\cW_2^2(\cN_d(m_\nu,\Sigma_\nu),\cJ)\ge \cW_2^2(\cN_d(m_\nu,\Sigma_\nu),\cN_d(m_\mu,\Sigma_\cJ))$$
                   and therefore $\cN_d(m_\mu,\Sigma_\cJ)$ is a $\cW_2$-projection $\cJ$ of $\cN_d(m_\nu,\Sigma_\nu)$ on the set of probability measures dominating $\cN_d(m_\mu,\Sigma_\mu)$ in the convex order. As a consequence, $$\cW_2^2(\cN_d(m_\nu,\Sigma_\nu),\cN_d(m_\mu,\Sigma_\cJ))=\inf_{\Sigma\in\cS^+_d:\Sigma_\mu\le\Sigma}\cW_2^2(\cN_d(m_\nu,\Sigma_\nu),\cN_d(m_\mu,\Sigma))$$ and, by \eqref{w2gauss}, $\Sigma_\cJ$ solves $\inf_{\Sigma\in\cS^+_d:\Sigma_\mu\le\Sigma}{\rm tr}\left(\Sigma-2(\Sigma_\nu^{1/2}\Sigma\Sigma_\nu^{1/2})^{1/2}\right)$.
Last, since the covariance matrix of any $\cW_2$-projection of $\nu$ and the set of probability measures dominating $\mu$ is greater than $\Sigma_\mu$, the inequality $$\inf_{\eta\in\cP_2(\R^d):\mu\le_{cx}\eta}\cW^2_2(\nu,\eta)\ge |m_\mu-m_\nu|^2+ \inf_{\Sigma\in\cS^+_d:\Sigma_\mu\le\Sigma}{\rm tr}\left(\Sigma_\nu+\Sigma-2(\Sigma_\nu^{1/2}\Sigma\Sigma_\nu^{1/2})^{1/2}\right)$$ follows from \eqref{minow2gaus} applied with $\mu$ replaced by some projection.           \end{proof}
                 In the next section, we discuss the uniqueness of $\cW_2$-projections of $\cN_d(0,\Sigma_\nu)$ on the set of probability measures dominating $\cN_d(0,\Sigma_\mu)$ in the convex order when $\Sigma_\nu$ is singular. We restrict the analysis to the case of centered distributions since $\eta\in\cP_2(\R^d)$ is a $\cW_2$-projection in this case if and only if the image of $\eta$ by $\R^d\ni x\mapsto x+m_\mu$ is a $\cW_2$-projection of $\cN_d(m_\nu,\Sigma_\nu)$ on the set of probability measures dominating $\cN_d(m_\mu,\Sigma_\mu)$ in the convex order.
                 \subsection{Discussion of the uniqueness of  $\cJ_2(\cN_d(0,\Sigma_\nu),\cN_d(0,\Sigma_\mu))$ when $\Sigma_\nu$ is singular}
                 When $\Sigma_\nu$ is singular, $\cN_d(0,\Sigma_\nu)$ is not absolutely continuous with respect to the Lebesgue measure and the uniqueness result for $\cJ_2(\cN_d(0,\Sigma_\nu),\cN_d(0,\Sigma_\mu))$ in \cite[Theorem 4.1]{AlCoJo20}  does not apply. Of course, for each $\eta\in\cP_2(\R^d)\setminus\{\cN_d(0,\Sigma_\mu)\}$ such that $\cN_d(0,\Sigma_\mu)\le_{cx}\eta$, $$\cW^2_2(\cN_d(0,0),\cN_d(0,\Sigma_\mu))=\int_{\R^d}|x|^2 \cN_d(0,\Sigma_\mu)(dx)<\int_{\R^d}|x|^2 \eta(dx)=\cW^2_2(\cN_d(0,0),\eta),$$
                 so that the unique $\cW_2$-projection of $\cN_d(0,0)$ on the set of probability measures dominating $\cN_d(0,\Sigma_\mu)$ in the convex order is $\cN_d(0,\Sigma_\mu)$.
In the next proposition, we characterize the $\cW_2$-projections of $\cN_d(0,\Sigma_\nu)$ when $\Sigma_\nu$ is singular but not equal to $0$ and check that there always is a unique projection that is Gaussian.

\begin{proposition}\label{propj2} Let $\Sigma_\mu,\Sigma_\nu\in{\mathcal S}^+_d$ with $\Sigma_\nu$ with rank $d'\in\{1,\cdots,d-1\}$ so that $\Sigma_\nu=O{\rm diag}(\lambda_1,\cdots,\lambda_{d'},0,\cdots,0)O^*$ for $\lambda_1,\cdots,\lambda_{d'}>0$ and $O\in{\mathcal O}_d$. Let $\Gamma\in{\mathcal S}_{d'}^+$ denote the covariance matrix of $\cJ_2(\cN_{d'}(0,{\rm diag}(\lambda_1,\cdots,\lambda_{d'})),\cN_{d'}(0,(O^*\Sigma_\mu O)_{1:d',1:d'}))$ and $\Sigma_\star\in\R^{d\times d}$ denote the matrix such that $(O^*\Sigma_\star O)_{1:d',1:d'}=\Gamma$ and all other entries of $O^*\Sigma_\star O$ are equal to those of $O^*\Sigma_\mu O$. Let $g:\R^d \to \R^{d'}$ be defined by $g(x)=O^*_{1:d',1:d}x$. 

   Then, $\eta_\star$ is a $\cW_2$-projection of $\cN_{d}(0,\Sigma_\nu)$ on the set of distributions dominating $\cN_{d}(0,\Sigma_\mu)$ in the convex order (i.e. $\eta_\star \in \arg \inf_{\cN_{d}(0,\Sigma_\mu)\le_{cx} \eta} \cW_2(\cN_{d}(0,\Sigma_\nu),\eta) $) if, and only if the covariance matrix of $\eta_\star$ is $\Sigma_\star$,  $g\#\eta_\star=\mathcal{N}_{d'}(0,\Gamma)$ and $\cN_{d}(0,\Sigma_\mu)\le_{cx} \eta_\star$.    
   The distribution $\cN_{d}(0,\Sigma_\star)$ is one of these projection and is thus the unique projection among Gaussian distributions. Moreover, $\Sigma_\star$ is the unique minimizer of $\inf_{\Sigma\in\cS^+_d:\Sigma_\mu\le\Sigma}{\rm tr}\left(\Sigma-2(\Sigma_\nu^{1/2}\Sigma\Sigma_\nu^{1/2})^{1/2}\right)$.
\end{proposition}
 Elaborating on this characterization and taking into account results derived in the next section (namely Corollary~\ref{cormgn} and Theorem~\ref{thmprojgauss}), we are able to give a necessary and sufficient condition for $\cN_{d}(0,\Sigma_\star)$ to be the unique projection in $\cP_2(\R^d)$ and not only among Gaussian distributions.\begin{proposition}\label{propcnsunic}
   Under the assumptions and notation of Proposition \ref{propj2}, $\cN_d(0,\Sigma_\star)$  is the unique $\cW_2$-projection of $\cN_{d}(0,\Sigma_\nu)$ on the set of distributions dominating $\cN_{d}(0,\Sigma_\mu)$ in the convex order if and only if ${\rm rk}(\Sigma_\star)={\rm rk}(\Sigma_\nu)$ or $\Sigma_\nu\le (\Sigma_\nu^{1/2}\Sigma_\mu\Sigma_\nu^{1/2})^{1/2}$.
\end{proposition}

 \begin{proof}[Proof of Proposition \ref{propj2}]
The definition of $\Sigma_\star$ clearly ensures that $O^*\Sigma_\mu O\le O^*\Sigma_\star O$ so that $\Sigma_\mu\le \Sigma_\star$ and $\cN_d(0,\Sigma_\mu)\le \cN_d(0,\Sigma_\star)$.
   
 Let $\eta\in\cP_2(\R^d)$. For $\R^d\ni x\mapsto f(x)=O^*x\in\R^{d}$, we have $|f(x)-f(y)|=|x-y|$ for all $x,y\in\R^d$ and therefore
$$\cW^2_2(\cN_d(0,\Sigma_\nu),\eta)=\cW^2_2(f\#\cN_d(0,\Sigma_\nu),f\#\eta)=\cW^2_2(\cN_d(0,{\rm diag}(\lambda_1,\cdots,\lambda_{d'},0,\cdots,0)),f\#\eta).$$
Since the image of any coupling $\pi$ between $\cN_d(0,{\rm diag}(\lambda_1,\cdots,\lambda_{d'},0,\cdots,0))$ and $f\#\eta$ by $\R^d\times\R^d\ni(x_1,\cdots,x_d,y_1,\cdots,y_d)\mapsto (x_1\cdots,x_{d'},y_1,\cdots,y_{d'})\in\R^{d'}\times\R^{d'}$ is some coupling between $\cN_{d'}(0,{\rm diag}(\lambda_1,\cdots,\lambda_{d'}))$ and $g\#\eta$, and $\pi(dx,dy)$ a.e. $|x-y|^2=|(x_1\cdots,x_{d'})-(y_1\cdots,y_{d'})|^2+\sum_{i=d'+1}^dy_i^2$, we deduce that \begin{align*}\cW^2_2(\cN_d(0,\Sigma_\nu),\eta)\ge  \cW^2_2(\cN_{d'}(0,{\rm diag}(\lambda_1,\cdots,\lambda_{d'})),g\#\eta)+\int_{\R^d}\sum_{i=d'+1}^d y_i^2f\#\eta(dy).\end{align*}
Let $f\#\eta_{y_{1:d'}}(dy_{d'+1:d})$ denote the conditional distribution of the last $d-d'$ variables given that the $d'$ first are equal to $y_{1:d'}$ under $f\#\eta$. Since for any coupling $\pi'$ between $\cN_{d'}(0,{\rm diag}(\lambda_1,\cdots,\lambda_{d'}))$ and $g\#\eta$,  $\pi'(dx_{1:d'},dy_{1:d'})\delta_0(dx_{d'+1:d})f\#\eta_{y_{1:d'}}(dy_{d'+1:d})$ is a coupling between $\cN_d(0,{\rm diag}(\lambda_1,\cdots,\lambda_{d'},0,\cdots,0))$ and $f\#\eta$, we even have
\begin{align}
  \cW^2_2(\cN_d(0,\Sigma_\nu),\eta)= \cW^2_2(\cN_{d'}(0,{\rm diag}(\lambda_1,\cdots,\lambda_{d'})),g\#\eta)+\int_{\R^d}\sum_{i=d'+1}^d y_i^2f\#\eta(dy).\label{decwassd'}\end{align}
We now suppose that $\cN_d(0,\Sigma_\mu)\le_{cx}\eta$. Then, by linearity of $f$ and $g$, $\cN_d(0,O^*\Sigma_\mu O)=f\#\cN_d(0,\Sigma_\mu)\le_{cx}f\#\eta$ and $\cN_{d'}(0,(O^*\Sigma_\mu O)_{1:d',1:d'})=g\#\cN_d(0,\Sigma_\mu)\le_{cx}g\#\eta$. We deduce that \begin{align}
   \int_{\R^d}\sum_{i=d'+1}^d y_i^2f\#\eta(dy)\ge \tr((O^*\Sigma_\mu O)_{d'+1:d,d'+1:d})=\tr((O^*\Sigma_\star O)_{d'+1:d,d'+1:d})% \sum_{i=d'+1}^d(O^*\Sigma_\mu O)_{ii}
,\label{opt1}\end{align}
 by convexity of $\R^d\ni (y_1,\cdots,y_d)\mapsto \sum_{i=d'+1}^d y_i^2$ and 
\begin{align}
   \cW^2_2(\cN_{d'}(0,{\rm diag}(\lambda_1,\cdots,\lambda_{d'})),g\#\eta)\ge\cW^2_2(\cN_{d'}(0,{\rm diag}(\lambda_1,\cdots,\lambda_{d'})),\cN_{d'}(0,\Gamma)),\label{opt2}
\end{align} by definition of $\cJ_2(\cN_{d'}(0,{\rm diag}(\lambda_1,\cdots,\lambda_{d'})),\cN_{d'}(0,(O^*\Sigma_\mu O)_{1:d',1:d'}))=\cN_{d'}(0,\Gamma)$. Plugging those two inequalities in the right-hand side of \eqref{decwassd'}, we obtain that
\begin{align*}\cW^2_2(\cN_d(0,\Sigma_\nu),\eta)&\ge \cW^2_2(\cN_{d'}(0,{\rm diag}(\lambda_1,\cdots,\lambda_{d'})),\cN_{d'}(0,\Gamma))+\tr((O^*\Sigma_\star O)_{d'+1:d,d'+1:d})\\&=\cW^2_2(\cN_{d'}(0,{\rm diag}(\lambda_1,\cdots,\lambda_{d'})),g\#\cN_{d}(0,\Sigma_\star))+\int_{\R^d}\sum_{i=d'+1}^d y_i^2f\#\cN_{d}(0,\Sigma_\star)(dy)
  \\&=\cW^2_2(\cN_d(0,\Sigma_\nu),\cN_d(0,\Sigma_\star)),
   \end{align*}
   where the last equality follows from \eqref{decwassd'} applied with $\eta$ replaced by $\cN_d(0,\Sigma_\star)$. Since $\cN_d(0,\Sigma_\mu)\le_{cx} \cN_d(0,\Sigma_\star)$, we conclude that $$\inf_{\eta\in\cP_2(\R^d):\cN_d(0,\Sigma_\mu)\le_{cx}\eta}\cW^2_2(\cN_d(0,\Sigma_\nu),\eta)=\cW^2_2(\cN_d(0,\Sigma_\nu),\cN_d(0,\Sigma_\star)).$$
   Moreover, a measure $\eta_\star\in\cP_2(\R^d)$ such that $\cN_d(0,\Sigma_\mu)\le_{cx}\eta_\star$ achieves the infimum if and only if the inequalities \eqref{opt2} and \eqref{opt1} are equalities, i.e. $g\#\eta_\star=\cN_{d'}(0,\Gamma)$ and $\int_{\R^d}\sum_{i=d'+1}^d y_i^2f\#\eta_\star(dy)=\tr((O^*\Sigma_\mu O)_{d'+1:d,d'+1:d})$. Therefore, the covariance matrix $\Sigma_{\eta_\star}$ of $\eta_\star$ must satisfy $O^*\Sigma_\mu O\le O^*\Sigma_{\eta_\star}O$, $(O^*\Sigma_{\eta_\star}O)_{1:d',1:d'}=\Gamma$ and $\tr((O^*\Sigma_{\eta_\star}O)_{d'+1:d,d'+1:d})=\tr((O^*\Sigma_\mu O)_{d'+1:d,d'+1:d})$. Since two ordered matrices in ${\mathcal S}^+_{d-d'}$ with the same trace are equal, we deduce that $(O^*\Sigma_{\eta_\star}O)_{d'+1:d,d'+1:d}=(O^*\Sigma_\mu O)_{d'+1:d,d'+1:d}$. With the inequality $O^*\Sigma_\mu O\le O^*\Sigma_{\eta_\star}O$, this implies that all the entries of $O^*\Sigma_\mu O$ and $O^*\Sigma_{\eta_\star}O$ with indices $ij$ such that $i\vee j\ge d'+1$ coincide. Since the equality $g\#\eta_\star=\cN_{d'}(0,\Gamma)$ implies $(O^*\Sigma_{\eta_\star}O)_{1:d',1:d'}=\Gamma$, we conclude that $O^*\Sigma_{\eta_\star}O=O^*\Sigma_\star O$ so that $\Sigma_{\eta_\star}=\Sigma_\star$. Conversely, it is clear that for a measure $\eta_\star \ge_{cx} \cN_{d}(0,\Sigma_\mu)$ with covariance matrix~$\Sigma_\star$ and  $g\#\eta_\star=\cN_{d'}(0,\Gamma)$, \eqref{opt1} and~\eqref{opt2} are equalities, which shows the necessary and sufficient condition for~$\eta_\star$ to be a $\cW_2$-projection. \end{proof}
\begin{proof}[Proof of Proposition \ref{propcnsunic}]
We first check the sufficient condition. Let $\eta_\star$ be a projection. When $\Sigma_\nu\le (\Sigma_\nu^{1/2}\Sigma_\mu\Sigma_\nu^{1/2})^{1/2}$, we have by Corollary \ref{cormgn} below that $\Sigma_\star=\Sigma_\mu$ so that, by Proposition \ref{propj2}, the covariance matrix of $\eta_\star$ is $\Sigma_\mu$. Since $\cN_d(0,\Sigma_\mu)\le_{cx}\eta_\star$ and two probability distributions ordered for the convex order and with the same covariance matrix are equal, 
we deduce that when $\Sigma_\nu\le (\Sigma_\nu^{1/2}\Sigma_\mu\Sigma_\nu^{1/2})^{1/2}$, then $\eta_\star=\cN_d(0,\Sigma_\mu)$. 

Let us now assume that ${\rm rk }(\Sigma_\star)=d'$. Then ${\rm rk }(O^*\Sigma_\star O)=d'$ and there exists $U\in\cO_d$ with columns $U_1,\cdots,U_d$ and $\theta_1,\cdots,\theta_{d'}>0$ such that $O^*\Sigma_\star O=U{\rm diag}(\theta_1,\cdots,\theta_{d'},0,\cdots,0)U^*$. Let $X$ be some centered $d$-dimensional random vector with covariance matrix $O^*\Sigma_\star O$. Since the covariance matrix of $U^*X$ is ${\rm diag}(\theta_1,\cdots,\theta_{d'},0,\cdots,0)$, we have ${\mathbb P}(U_{d'+1}^*X=\cdots=U_{d}^*X=0)=1$ and, since $X=UU^*X$, we deduce that $X=\sum_{i=1}^{d'}(U_i^*X)U_i$. Let for $i\in\{1,\cdots,d'\}$, $\tilde U_i$ denote the vector in $\R^{d'}$ with components equal to the $d'$ first components of $U_i$. The equality of the covariance matrices of both sides of the equality $X_{1:d'}=\sum_{i=1}^{d'}(U_i^*X)\tilde U_i$, writes $\Gamma=\sum_{i=1}^{d'}\theta_i \tilde U_i\tilde U_i^*$. By Theorem~\ref{thmprojgauss}, the rank of $\Gamma$ is $d'$. We deduce that $\bigcap_{i=1}^{d'}\tilde U_i^\perp=\{0\}$ so that the kernel of $\left(\begin{array}{c}\tilde U_1^*
   \\\vdots\\\tilde U_{d'}^*
\end{array}\right)\in\R^{d'\times d'}$ is $\{0\}$ and this matrix has rank $d'$. Hence $\tilde U_1,\cdots,\tilde U_{d'}$ is a basis of $\R^{d'}$. Therefore there exists $A\in\R^{d\times d'}$ such that $X=AX_{1:d'}$. If $Y\sim \eta_\star$, then by Proposition \ref{propj2}, $(O^*Y)_{1:d'}\sim\cN_{d'}(0,\Gamma)$ and the covariance matrix of $O^*Y$ is $O^*\Sigma_\star O$ so that $O^*Y=A(O^*Y)_{1:d'}$. Therefore $O^*Y$ and $Y$ are Gaussian and $\eta_\star=\cN_d(0,\Sigma_\star)$.\\

We now turn to the proof of the necessary condition and suppose that ${\rm rk}(\Sigma_\star)\ne d'$ and $\Sigma_\nu\not\le (\Sigma_\nu^{1/2}\Sigma_\mu\Sigma_\nu^{1/2})^{1/2}$. Then since, by Theorem~\ref{thmprojgauss} below, ${\rm rk}(\Gamma)=d'$, ${\rm rk}(\Sigma_\star)>d'$ and, by Corollary \ref{cormgn} below,  $\Sigma_\star \ne \Sigma_\mu$, i.e.  $(O^*\Sigma_\mu O)_{1:d',1:d'}\ne \Gamma$. We distinguish two cases.

  \paragraph{\bf Case ${\rm rk}(O^*\Sigma_\mu O)>{\rm rk}((O^*\Sigma_\mu O)_{1:d',1:d'})$:} When $\tilde \xi\in\R^{d'}$ is such that $\tilde \xi^*(O^*\Sigma_\mu O)_{1:d',1:d'}\tilde\xi=0$, then for $W$ centered $d$-dimensional random vector with covariance matrix $O^*\Sigma_\mu O$, $\tilde \xi^*W_{1:d'}=0$ so that $0=\E[\tilde\xi^*W_{1:d'}W^*_{d'+1:d}]=\tilde\xi^*O^*\Sigma_\mu O_{1:d',d'+1:d}$. Therefore ${\rm rk}((O^*\Sigma_\mu O)_{1:d',1:d})={\rm rk}((O^*\Sigma_\mu O)_{1:d',1:d'})$ and there exists $\xi\in\R^d$ such that $(O^*\Sigma_\mu O\xi)_{1:d'}=0$ and $\xi^*O^*\Sigma_\mu O\xi>0$. Since by definition   $(O^*\Sigma_\mu O)_{1:d',1:d'}\le \Gamma$ and by assumption $(O^*\Sigma_\mu O)_{1:d',1:d'}\ne \Gamma$, there exists an eigenvector $\tilde{\chi}\in\R^{d'}$ with $|\tilde{\chi}|=1$ of $\Gamma-(O^*\Sigma_\mu O)_{1:d',1:d'}$ associated with a positive eigenvalue $\theta>0$. We can take $\lambda \in (0,1)$ such that 
  $$ \lambda^2 +\lambda \xi^*_{1:d'} \tilde{\chi}\ne 0. $$
  We then set $\chi=\lambda\tilde{\chi}$, so that $ |\chi|^2 +\xi^*_{1:d'} \chi\ne 0$. By construction, we have
$\Gamma-(O^*\Sigma_\mu O)_{1:d',1:d'}-\theta \chi\chi^*\ge 0$ % and $\tilde \chi\in\R^d$ be such that $\tilde \chi_{1:d'}=\chi$
. For $X,\varepsilon,Z$ independent with $X\sim\cN_d(0,\Sigma_\mu)$, ${\mathbb P}(\varepsilon=\pm 1)=\frac 12$ and $Z\sim\cN_{d'}(0,\Gamma-(O^*\Sigma_\mu O)_{1:d',1:d'}-\theta \chi\chi^*)$, we define $$Y=X+O_{1:d,1:d'} Z+\left(\varepsilon \sqrt{\frac{\theta}{\xi^*O^*\Sigma_\mu O\xi}}\xi^*O^*X\right)O_{1:d,1:d'} \chi.$$
The random vectors $X$, $Z$ and the scalar random variable $\varepsilon \frac{1}{\sqrt{\xi^*O^*\Sigma_\mu O\xi}} \xi^*O^*X\sim\cN_1(0,1)$ are not correlated so that the covariance matrix of $Y$ is equal to
$$\Sigma_\mu+O_{1:d,1:d'}\left(\Gamma-(O^*\Sigma_\mu O)_{1:d',1:d'}-\theta \chi\chi^*
  +\theta \chi\chi^*\right)O^*_{1:d',1:d}=\Sigma_\star,$$
where we have used the definition of~$\Sigma_\star$ and that for $M\in \R^{d \times d}$ such that $M_{ij}=1_{\{i\vee j\le d'\}}\left(\Gamma_{ij}-(O^*\Sigma_\mu O)_{ij}\right)$, we have  $OMO^*=O_{1:d,1:d'} \left(\Gamma_{ij}-(O^*\Sigma_\mu O)_{ij}\right)O^*_{1:d',1:d}$. 
Moreover, by independence of $(Z,\varepsilon)$ with $X$,
$$\E[Y|X]=X+O_{1:d,1:d'}\left(\E[Z]+\E[\varepsilon]\sqrt{\frac{\theta}{\xi^*O^*\Sigma_\mu O\xi}}\xi^*O^*X\chi\right)=X,$$
so that $\cN_d(0,\Sigma_\mu)={\mathcal L}(X)\le_{cx}{\mathcal L}(Y)$. To show the optimality of ${\mathcal L}(Y)$, according to Proposition \ref{propj2}, it remains to prove that $g\#{\mathcal L}(Y)=\mathcal{N}(0,\Gamma)$. Since $Y$ has the right covariance matrix, it is sufficient to check that  $(O^*Y)_{1:d'}$ is a Gaussian random vector.
Since $(O^*\Sigma_\mu O\xi)_{1:d'}=0$, $(O^*X)_{1:d'}$ is not correlated with $\xi^*O^*X$  so that, since $X$ is Gaussian, $\xi^*O^*X$ is independent of $(O^*X)_{1:d'}$ and so is $(Z,\varepsilon\xi^*O^*X)$. As a consequence, the random vector $((O^*X)_{1:d'},Z,\varepsilon\xi^*O^*X)$ with three independent Gaussian components is Gaussian and so is $(O^*Y)_{1:d'}=(O^*X)_{1:d'}+Z+\varepsilon \sqrt{\frac{\theta}{\xi^*O^*\Sigma_\mu O\xi}}\xi^*O^*X\chi$. We conclude that ${\mathcal L}(Y)$ is a projection of $\cN_d(0,\Sigma_\nu)$ on the set of probability measures dominating $\cN_d(0,\Sigma_\mu)$ in the convex order. Last, \begin{align*}
&(\chi^*O^*_{1:d',1:d}+\xi^*O^*)Y\\&=\chi^*O^*_{1:d',1:d}X+\xi^*O^*X\left(1+\varepsilon \sqrt{\frac{\theta}{\xi^*O^*\Sigma_\mu O\xi}}(|\chi|^2+\xi_{1:d'}^*\chi)\right)+(\chi^*+\xi_{1:d'}^*)Z
\end{align*}
is not Gaussian since the absolute value of the factor of $\xi^*O^*X$ is not constant. Hence ${\mathcal L}(Y)\ne \cN_d(0,\Sigma_\star)$.

\paragraph{\bf Case ${\rm rk}(O^*\Sigma_\mu O)={\rm rk}((O^*\Sigma_\mu O)_{1:d',1:d'})$ :}
Let $X$ be some centered $d$-dimensional random vector with covariance matrix $O^*\Sigma_\mu O$. Then, repeating with $\Sigma_\star$ replaced by $\Sigma_\mu$ the above reasoning leading to the uniqueness of the projection when ${\rm rk}(\Sigma_\star)=d'$, we obtain that $X_{d'+1:d}=BX_{1:d'}$ for some deterministic matrix $B\in\R^{(d-d')\times d'}$. Let us now prove by contradiction the existence of $\tilde \xi\in\R^{d'}$  such that \begin{equation}
   \label{existsxi}\tilde \xi^*(\Gamma-(O^*\Sigma_\mu O)_{1:d',1:d'})\tilde\xi>0\mbox{ and }(O^*\Sigma_\mu O)_{d'+1:d,1:d'}\tilde\xi\ne 0.
\end{equation}Let $\tilde U\in\cO_{d'}$ be such that $\Gamma-(O^*\Sigma_\mu O)_{1:d',1:d'}=\tilde U{\rm diag}(\tilde \theta_1,\cdots,\tilde \theta_{r},0,\cdots,0)\tilde U^*$ with $r\in\{1,\cdots,d'\}$ and $\tilde\theta_1\ge \cdots\ge \tilde\theta_{r}>0$. We suppose that  $(O^*\Sigma_\mu O)_{d'+1:d,1:d'}\tilde U_{1:d',1:r}=0$, which amounts to supposing that $\tilde\xi$ does not exist.  Since $X_{d'+1:d}$ is not correlated with $(\tilde U^*)_{1:r,1:d'}X$, we even have $X_{d'+1:d}=
\tilde B(\tilde U^*)_{r+1:d',1:d'}X_{1:d'}$ for some deterministic matrix $\tilde B\in\R^{(d-d')\times (d'-r)}$. For any centered $d$-dimensional random vector $Z$ with covariance matrix $O^*\Sigma_\star O$, the covariance matrix of $((\tilde U^*)_{r+1:d',1:d'}Z,Z_{d'+1:d})$ is the same as the one of $((\tilde U^*)_{r+1:d',1:d'}X,X_{d'+1:d})$ so that $Z_{d'+1:d}=
\tilde B(\tilde U^*)_{r+1:d',1:d'}Z_{1:d'}$. We deduce that the rank of the covariance matrix of the random vector $(\tilde U^* Z_{1:d'},Z_{d'+1:d})$ is not greater than $d'$. Since this random vector is the image of $Z$ by some orthogonal transformation, $d'\ge {\rm rk}(O^*\Sigma_\star O)={\rm rk}(\Sigma_\star)$, which provides the desired contradiction. 
Setting $\Theta=(O^*\Sigma_\mu O)_{1:d',1:d'}$, we now recall the martingale coupling  introduced in the second point of Remark 1.17 \cite{joupagm} that sends ${\cN}_{d'}(0,\Theta)$ to ${\cN}_{d'}(0,\Gamma)$ but does not consist in the addition of an independent random vector distributed according to ${\cN}_{d'}(0,\Gamma-\Theta)$. 
  The inequality $\Theta\le \Gamma$ implies with the diagonalization of $\Gamma^{-1/2}\Theta\Gamma^{-1/2}$  that $\Theta=\Gamma^{1/2}V{\rm diag}(\zeta^2_1,\cdots,\zeta^2_{d'})V^*\Gamma^{1/2}$ for some $V\in{\mathcal O}(d')$ and $\zeta_1,\cdots,\zeta_{d'}\in [0,1]$. 
The image of $\cN_{d'}(0,I_{d'})$ by the Markov kernel \begin{align*}
   Q(w,d\tilde w)&=\otimes_{i=1}^{d'}\left(\frac{1-\zeta_i}{2}\delta_{-\zeta_iw_i}+\frac{1+\zeta_i}{2}\delta_{\zeta_iw_i}\right)(d\tilde w_i)\\&=\sum_{(s_1,\cdots,s_{d'})\in\{-1,1\}^{d'}}\prod_{i=1}^{d'}\frac{1+s_i\zeta_i}{2}\delta_{(\zeta_1s_1 w_1,\cdots,\zeta_{d'}s_{d'} w_{d'})}(d\tilde w)
\end{align*}
is $\cN_{d'}(0,{\rm diag}(\zeta_1^2,\cdots,\zeta_{d'}^2))$. The image of $\cN_{d'}(0,I_{d'})(dw)Q(w,d\tilde w)$ by $\R^{d'}\times\R^{d'}\ni (w,\tilde w)\mapsto (\Gamma^{1/2}V\tilde w,\Gamma^{1/2}Vw)\in\R^{d'}\times\R^{d'}$ is a martingale coupling $\pi(dx,d\tilde x)=\cN_{d'}(0,\Theta)(dx)\pi_x(d\tilde x)$ between $\cN_{d'}(0,\Theta)$ and $\cN_{d'}(0,\Gamma)$.  
Let now $X\sim\cN_d(0,O^*\Sigma_\mu O)$, $\tilde X$ be distributed according to $\pi_{X_{1:d'}}(d\tilde x)$ and $Y$ denote the $\R^d$-valued random vector such that $Y_{1:d'}=\tilde X$ and $Y_{d'+1:d}=BX_{1:d'}=X_{d'+1:d}$. One has $\E[\tilde X|X]=\E[\E[\tilde X|X_{1:d'}]|X]=X_{1:d'}$ by the martingale property so that $\E[Y|X]=\left( \begin{array}{c}
  X_{1:d'}\\
  BX_{1:d'}
\end{array}\right)=X$ and thus  $\E[OY|OX]=OX$. Moreover $\E[X_{d'+1:d}\tilde X^*]=\E[X_{d'+1:d}\E[\tilde X|X]^*]=\E[X_{d'+1:d}X_{1:d'}^*]$ so that the covariance matrix of $Y$ is $O^*\Sigma_\star O$. Then, the covariance matrix of $OY$ is $\Sigma_\star$, $(O^*OY)_{1:d'}=Y_{1:d'}=\tilde X\sim\cN_{d'}(0,\Gamma)$ and $\cN_d(0,\Sigma_\mu)={\mathcal L}(OX)\le_{cx} {\mathcal L}(OY)$ by Strassen's theorem. As a consequence, by Proposition \ref{propj2},  ${\mathcal L}(OY)$ is a $\cW_2$-projection of $\cN_d(0,\Sigma_\nu)$ on the set of probability measures dominating $\cN_d(0,\Sigma_\mu)$ in the convex order. For $\tilde\xi\in\R^{d'}$ such that \eqref{existsxi} holds, the conditional law of $\tilde \xi^* Y_{1:d'}$ given $Y_{d'+1:d}$ is not Gaussian: indeed, $(O^*\Sigma_\mu O)_{d'+1:d,1:d'}\tilde\xi\ne 0$ gives that $Y_{d'+1:d}=X_{d'+1:d}$ is not independent of $X_{1:d'}$ and $\tilde \xi^*(\Gamma-(O^*\Sigma_\mu O)_{1:d',1:d'})\tilde\xi>0$ gives that $\tilde \xi^* Y_{1:d'}$ is distinct from $\tilde \xi^* X_{1:d'}$. Therefore ${\mathcal L}(Y)$ is not Gaussian. Hence ${\mathcal L}(OY)\ne \cN_d(0,\Sigma_\star)$. 
  \end{proof} 
\begin{example}
  Let $\nu=\cN_2\left(0,\left(\begin{array}{cc}2 &0\\0 & 0\end{array}\right)\right)$. If $\mu=\cN_2\left(0,I_2\right)$, then $\cN_2\left(0,\left(\begin{array}{cc}2 &0\\0 & 1\end{array}\right)\right)$ and $\eta(dx_1,dx_2)=\int_{(y_1,y_2)\in\R^2} \frac 12\left(\delta_{(y_1-y_2,y_2)}+\delta_{(y_1+y_2,y_2)}\right)(dx_1,dx_2)\mu(dy_1,dy_2)$ both 
  are $\cW_2$ projections of $\nu$ on the set of probability measures dominating $\mu$ in the convex order.
 If $\mu=\cN_2\left(0,\left(\begin{array}{cc}1 &1\\1 & 1\end{array}\right)\right)$ then $\cN_2\left(0,\left(\begin{array}{cc}2 &1\\1 & 1\end{array}\right)\right)$ and $\eta(dx_1,dx_2)=\int_{(y_1,y_2)\in\R^2} \left(\frac{\sqrt{2}-1}{2\sqrt{2}}\delta_{(-\sqrt{2}y_1,y_2)}+\frac{\sqrt{2}+1}{2\sqrt{2}}\delta_{(\sqrt{2}y_1,y_2)}\right)(dx_1,dx_2)\mu(dy_1,dy_2)$ both 
  are $\cW_2$ projections of $\nu$ on the set of probability measures dominating $\mu$ in the convex order. The example $\mu=\cN_2\left(0,I_2\right)$ is prototypical of the case ${\rm rk}(O^*\Sigma_\mu O)>{\rm rk}((O^*\Sigma_\mu O)_{1:d',1:d'})$ in the previous proof while the example $\mu=\cN_2\left(0,\left(\begin{array}{cc}1 &1\\1 & 1\end{array}\right)\right)$ is prototypical of the case ${\rm rk}(O^*\Sigma_\mu O)={\rm rk}((O^*\Sigma_\mu O)_{1:d',1:d'})$.
\end{example}

\section{Characterization of the Wasserstein projections for Gaussian distributions}

In this section, we only work with centered Gaussian distributions. We thus use the following shorthand notation, for $\Sigma_\mu,\Sigma_\nu \in \cS_d^+$ 
\begin{align*}
  &\BW^2(\Sigma_\mu,\Sigma_\nu)=\cW^2_2(\cN_d(0,\Sigma_\mu),\cN_d(0,\Sigma_\nu))=\tr\left(\Sigma_\mu+\Sigma_\nu-2(\Sigma_\mu^{1/2}\Sigma_\nu\Sigma_\mu^{1/2})^{1/2}\right),\\
  &\cI_2(\Sigma_\mu,\Sigma_\nu)\in \cS_d^+ \text{ is such that } \cI_2(\cN_d(0,\Sigma_\mu),\cN_d(0,\Sigma_\nu))=\cN_d(0,\cI_2(\Sigma_\mu,\Sigma_\nu)),\\
  &\cJ_2(\Sigma_\nu,\Sigma_\mu)\in \cS_d^+ \text{ is such that } \cJ_2(\cN_d(0,\Sigma_\nu),\cN_d(0,\Sigma_\mu))=\cN_d(0,\cJ_2(\Sigma_\nu,\Sigma_\mu)).
\end{align*} 
The distance $\BW$ is called the Bures-Wasserstein distance on semidefinite positive matrices, see e.g.~\cite{BhJaLi19}. Note that the existence and uniqueness of $\cJ_2(\Sigma_\nu,\Sigma_\mu)$ (resp. $\cI_2(\Sigma_\mu,\Sigma_\nu)$) is guaranteed by Proposition~\ref{propj2} (resp.~\ref{prop_projG}).

We now state the main result of this section.

\begin{theorem}\label{thmprojgauss}                         Let $\Sigma_\mu,\Sigma_\nu\in\cS^+_d$. There exists $O\in{\mathcal O}_d$ such that $DO^*\Sigma_\mu OD\le O^*\Sigma_\nu O$ where $D={\rm diag}\left(1\wedge\sqrt{\frac{(O^*\Sigma_\nu O)_{11}}{(O^*\Sigma_\mu O)_{11}}},\cdots,1\wedge\sqrt{\frac{(O^*\Sigma_\nu O)_{dd}}{(O^*\Sigma_\mu O)_{dd}}}\right)$ (here and afterwards, we use the convention that $1\wedge\sqrt{\frac{(O^*\Sigma_\nu O)_{ii}}{(O^*\Sigma_\mu O)_{ii}}}=1$ when $(O^*\Sigma_\mu O)_{ii}=0$). Moreover, for any such $O\in{\mathcal O}_d$,\begin{align}\cI_2(\Sigma_\mu,\Sigma_\nu)&=ODO^*\Sigma_\mu ODO^*,
  \notag\\
 \BW^2(\Sigma_\mu,\cI_2(\Sigma_\mu,\Sigma_\nu))
&=\sum_{i=1}^d\left(\sqrt{(O^*\Sigma_\mu O)_{ii}}-\sqrt{(O^*\Sigma_\nu O)_{ii}}\right)_+^2=\BW^2(\Sigma_\nu,\cJ_2(\Sigma_\nu,\Sigma_\mu)),\notag\\\cJ_2(\Sigma_\nu,\Sigma_\mu)&=O\tilde \Sigma_{\cJ}O^*\mbox{ where }\notag\\(\tilde \Sigma_{\cJ})_{ij}=&1_{\{(O^*\Sigma_\nu O)_{ii}(O^*\Sigma_\nu O)_{jj}=0\}}(O^*\Sigma_\mu O)_{ij}+1_{\{(O^*\Sigma_\nu O)_{ii}(O^*\Sigma_\nu O)_{jj}>0\}}\frac{(O^*\Sigma_\nu O)_{ij}}{D_{ii}D_{jj}}\label{deftsigj}.\end{align}
\end{theorem}
\begin{remark}\label{rk_thm_main}
  \begin{itemize}
  \item One has $\tr(\cI_2(\Sigma_\mu,\Sigma_\nu))=\tr(DO^*\Sigma_\mu OD)=\sum_{i=1}^d (O^*\Sigma_\mu O)_{ii}\wedge (O^*\Sigma_\nu O)_{ii}$ while $\tr(\cJ_2(\Sigma_\nu,\Sigma_\mu))=\sum_{i=1}^d (O^*\Sigma_\mu O)_{ii}\vee (O^*\Sigma_\nu O)_{ii}$. As a consequence,
    \begin{align*}
     \tr(\cI_2(\Sigma_\mu,\Sigma_\nu))+\tr(\cJ_2(\Sigma_\nu,\Sigma_\mu))&=\sum_{i=1}^d (O^*\Sigma_\mu O)_{ii}+\sum_{i=1}^d (O^*\Sigma_\nu O)_{ii}\\&=\tr(O^*\Sigma_\mu O)+\tr(O^*\Sigma_\nu O)=\tr(\Sigma_\mu)+\tr(\Sigma_\nu). 
    \end{align*}
    We deduce that the sum of the integrals of the function $\R^d\ni x\mapsto |x|^2$ againsts $\cI_2(\cN_d(m_\mu,\Sigma_\mu),\cN_d(m_\nu,\Sigma_\nu))$ and $\cJ_2(\cN_d(m_\nu,\Sigma_\nu),\cN_d(m_\mu,\Sigma_\mu))$ is equal to the sum of its integrals against $\cN_d(m_\mu,\Sigma_\mu)$ and $\cN_d(m_\nu,\Sigma_\nu)$. Note that the same property holds in dimension $d=1$ even for non Gaussian probability measures
    $$\forall  \mu,\nu\in\cP_2(\R),\;\int_\R x^2(\cI(\mu,\nu)+\cJ(\nu,\mu))(dx)=\int_\R x^2(\mu+\nu)(dx).$$
This follows from \eqref{eq:quantileProjections}
combined with the equalities $\int_\R x^2\eta(dx)=\int_0^1(F_\eta^{-1}(u))^2du$, consequence of the inverse transform sampling and $\int_0^1 \partial_-G(u)\partial_-\co(G)(u)du=\int_0^1(\partial_-\co(G)(u))^2du$ (see Theorem 5 \cite{CLRM}). Still in dimension $1$, $$W_2^2(\mu,\cJ(\nu,\mu))=\int_0^1(\partial_- G(u)-\partial_-\co(G)(u))^2du=W_2^2(\nu,\cI(\mu,\nu))$$
equality which is preserved in the multidimensional Gaussian case when one can choose the orthogonal matrix $O$ in Theorem \ref{thmprojgauss}  such that $O^*\Sigma_\mu O$ and $O^*\Sigma_\nu O$ share the same correlation matrix (see Corollary~\ref{cor_DDhat} below for a discussion of this setting) so that $$\BW^2(\Sigma_\mu,\cJ_2(\Sigma_\nu,\Sigma_\mu))=\sum_{i=1}^d\left(\sqrt{(O^*\Sigma_\nu O)_{ii}}-\sqrt{(O^*\Sigma_\mu O)_{ii}}\right)_+^2=\BW^2(\Sigma_\nu,\cI_2(\Sigma_\mu,\Sigma_\nu)).$$
\item When $\Sigma_\mu,\Sigma_\nu\in\cS^+_d$ commute, then $O^*\Sigma_\mu O={\rm diag}(\mu_1,\cdots,\mu_d)$ and $O^*\Sigma_\nu O={\rm diag}(\nu_1,\cdots,\nu_d)$ for some $O\in\cO_d$ and $\mu_1,\cdots,\mu_d,\nu_1,\cdots,\nu_d\in\R_+$. The associated diagonal matrix is $D={\rm diag}\left(1\wedge \sqrt{\frac{\nu_1}{\mu_1}},\cdots,1\wedge\sqrt{\frac{\nu_d}{\mu_d}}\right)$ and since $DO^*\Sigma_\mu OD={\rm diag}(\mu_1\wedge \nu_1,\cdots,\mu_d\wedge \nu_d)\le O^*\Sigma_\nu O$, Theorem \ref{thmprojgauss} implies that \begin{align*}&\cI_2(\Sigma_\mu,\Sigma_\nu)=O{\rm diag}(\mu_1\wedge \nu_1,\cdots,\mu_d\wedge \nu_d)O^*,\\&\cJ_2(\Sigma_\nu,\Sigma_\mu)=O{\rm diag}(\mu_1\vee \nu_1,\cdots,\mu_d\vee\nu_d)O^*,\end{align*} and $\BW^2(\Sigma_\mu,\cI_2(\Sigma_\mu,\Sigma_\nu))=\sum_{i=1}^d\left(\sqrt{\mu_i}-\sqrt{\nu_i}\right)_+^2$.

\item From Theorem~\ref{thmprojgauss}, we get that $\R^d\ni x\mapsto O^*DOx$ is the optimal transport map  between $\cN_d(0,\Sigma_\mu)$ and $\cN_d(0,\cI_2(\Sigma_\mu,\Sigma_\nu))$. It is the gradient of $ x\mapsto \frac 12 x^* O^*DOx$ which is a convex contraction, in line with~\cite[Theorem 7.4]{KimRuan}.
 When in addition $(O^*\Sigma_\nu O)_{ii}>0$ for all $i\in\{1,\cdots,d\}$ and in particular when $\Sigma_\nu$ is positive definite, then $D$ is positive definite and $O\tilde\Sigma_\cJ O^*=OD^{-1}O^*\Sigma_\nu OD^{-1}O^*$. The optimal transport map  between $\cN_d(0,\Sigma_\nu)$ and $\cN_d(0,\cJ_2(\Sigma_\nu,\Sigma_\mu))$ is $\R^d\ni x\mapsto OD^{-1}O^* x$. It is a convex expansion as expected from~\cite[Theorem 7.6]{KimRuan}, and it is the inverse of the optimal transport map between $\cN_d(0,\Sigma_\mu)$ and $\cN_d(0,\cI_2(\Sigma_\mu,\Sigma_\nu))$, as given by~\cite[Corollary 8.5]{KimRuan}.  Let us note however that we only need here that $(O^*\Sigma_\nu O)$ has positive diagonal coefficients while the application of \cite[Theorem 7.6 and Corollary 8.5]{KimRuan} would require to have $\Sigma_\nu \in \cS_d^{+,*}$. 
\item Still when $(O^*\Sigma_\nu O)_{ii}>0$ for all $i\in\{1,\cdots,d\}$, we have equality of  the images of the optimal couplings between $\cN_d(0,\Sigma_\mu)$ and $\cN_d(0,\cI_2(\Sigma_\mu,\Sigma_\nu))$ and between $\cN_d(0,\Sigma_\nu)$ and $\cN_d(0,\cJ_2(\Sigma_\nu,\Sigma_\mu))$ by $\R^{d}\times\R^d\ni(x,y)\mapsto y-x$. These images are centered Gaussian distributions and their respective covariance matrices $(I_d-ODO^*)\Sigma_\mu(I_d-ODO^*)$ and $(I_d-OD^{-1} O^*)\Sigma_\nu(I_d-OD^{-1}O^*)$ are equal, which strengthens the equality of the traces which follows from \eqref{fb}.
Indeed $(I_d-D)O^*\Sigma_\mu O(I_d-D)\le (I_d-D)D^{-1}O^*\Sigma_\nu OD^{-1}(I_d-D)$ so that, with the equality of the traces, \begin{align*}
(I_d-D)O^*\Sigma_\mu O(I_d-D)=(D^{-1}-I_d)O^*\Sigma_\nu O(D^{-1}-I_d).
\end{align*}

\end{itemize}
\end{remark}

\subsection{Some properties of the Bures-Wasserstein distance}

We first present useful results on the Bures-Wasserstein distance, and refer to the companion paper~\cite{AJ_arxiv} for further details. 
We introduce the set of correlation matrices by $\mathfrak{C}_d= \{ C \in \cS^+_d : \forall i, C_{ii}=1 \}$. For $\Sigma\in \cS_d^+$, we 
define $$\textup{dg}(\Sigma)={\rm diag}(\Sigma_{11},\dots,\Sigma_{dd})\in S_d^+$$ the diagonal matrix obtained with the diagonal elements of $\Sigma$.
We will say that two matrices $\Sigma_1,\Sigma_2$ share the correlation matrix $C\in \mathfrak{C}_d$ if $$\Sigma_i =\textup{dg}(\Sigma_i)^{1/2} C\textup{dg}(\Sigma_i)^{1/2} , \ i \in \{1,2\}.$$
In this case, an optimal $\cW_2$ coupling between $\mathcal{N}_d(0,\Sigma_1)$ and $\mathcal{N}_d(0,\Sigma_2)$ can be easily constructed as follows. Let $Z\sim  \mathcal{N}_d(0,C)$ and define $X=\sqrt{{\rm dg}(\Sigma_1)}Z$, $Y=\sqrt{{\rm dg}(\Sigma_2)}Z$. Then, we have $X\sim \mathcal{N}_d(0,\Sigma_1)$, $Y\sim \mathcal{N}_d(0,\Sigma_2)$ and
\begin{align*}
\E[|X-Y|^2]=\sum_{i=1}^d   (\sqrt{(\Sigma_1)_{ii}}-\sqrt{(\Sigma_2)_{ii}})^2 \E[Z_{ii}^2]=\sum_{i=1}^d   \left(\sqrt{(\Sigma_1)_{ii}}-\sqrt{(\Sigma_2)_{ii}}\right)^2.
\end{align*} 
We check by calculations that the latter sum equals $\tr\left(\Sigma_1+\Sigma_2-2(\Sigma_1^{1/2}\Sigma_2\Sigma_1^{1/2})^{1/2}\right)$, which gives the optimality. 

In the general case, we first notice that $\BW^2(\Sigma_1,\Sigma_2)=\BW^2(O^*\Sigma_1 O,O^*\Sigma_2O)$ for any $O\in \mathcal{O}_d$, and show (see~\cite{AJ_arxiv}) that we can find $O$ such that $O^*\Sigma_1 O$ and $O^*\Sigma_2 O$ share the same correlation matrix. When $\Sigma_1\in \mathcal{S}^{+,*}_d$,  one can choose $O$ such that \begin{equation}\label{O_share_correl}
  O^*\Sigma_1^{-1/2}(\Sigma_1^{1/2}\Sigma_2 \Sigma_1^{1/2})^{1/2}\Sigma_1^{-1/2}O \text{ is diagonal.}\end{equation} This leads to the following theorem.
\begin{theorem}\label{thm_BW}
  Let $\Sigma_1,\Sigma_2 \in \mathcal{S}^+_d$. Then, there exists $O\in \mathcal{O}_d$ and $C \in \mathfrak{C}_d$ such that $O^*\Sigma_1 O$ and $O^*\Sigma_2 O$ share the correlation matrix~$C$. Besides, for any such $(O,C)$, we have 
  $$\BW^2(\Sigma_1,\Sigma_2)=\sum_{i=1}^d   \left(\sqrt{(O^*\Sigma_1 O)_{ii}}-\sqrt{(O^*\Sigma_2 O)_{ii}}\right)^2.$$
\end{theorem}
We now state a Lemma that will be used in some calculations.
\begin{lemma}\label{lem_calcD}
  Let $\Sigma_1,\Sigma_2 \in \cS_d^{+}$. Let $O\in \mathcal{O}_d$ and $C \in \mathfrak{C}_d$ such that $O^*\Sigma_1 O$ and $O^*\Sigma_2 O$ share the correlation matrix~$C$. Let $D$ be the diagonal matrix such that $D_{ii}=\sqrt{\frac{(O^*\Sigma_2 O)_{ii}}{(O^*\Sigma_1 O)_{ii}}}$ if $(O^*\Sigma_1 O)_{ii}>0$ and $D_{ii}=1$ otherwise. Then, we have
  \begin{equation}
   (O^*\Sigma_1 O)^{\frac 12} D(O^*\Sigma_1 O)^{\frac 12} = O^* (\Sigma_1^{\frac 12} \Sigma_2 \Sigma_1^{\frac 12})^{\frac 12}  O.\label{egqd}
  \end{equation}
    If $\Sigma_1 \in \cS_d^{+,*}$, we furthermore have
  \begin{equation}
   D=O^* \Sigma_1^{- \frac 12}\left(\Sigma_1^{\frac 12} \Sigma_2 \Sigma_1^{\frac 12} \right)^{\frac 12}  \Sigma_1^{- \frac 12} O=\textup{dg}(O^*\Sigma_2 O)^{\frac 12}\textup{dg}(O^*\Sigma_1 O)^{- \frac 12}.\label{egqd2}
  \end{equation}
\end{lemma}
\begin{proof}
Since $O^*\Sigma_1 O= \textup{dg}(O^*\Sigma_1 O)^{\frac 12}C\textup{dg}(O^*\Sigma_1 O)^{\frac 12}$, we have $DO^*\Sigma_1 OD=P_1 O^*\Sigma_2 O P_1$, where $P_1$ is the diagonal matrix such that $(P_1)_{ii}=1$ if $(O^*\Sigma_1 O)_{ii}>0$ and $(P_1)_{ii}=0$ otherwise. Let $K_1=\{i : (O^*\Sigma_1 O)_{ii}=0\}$.
We have $(O^*\Sigma_1 O)_{ij}=0$ for $i,j \in K_1$ and thus $(O^*\Sigma_1 O)^{\frac 12}_{ij}=0$ for $i,j \in K_1$. This gives  $(O^*\Sigma_1 O)^{\frac 12} P_1=(O^*\Sigma_1 O)^{\frac 12}=O^*\Sigma^{\frac 12}_1 O $ and  $ P_1 (O^*\Sigma_1 O)^{\frac 12} =(O^*\Sigma_1 O)^{\frac 12} =O^*\Sigma^{\frac 12}_1 O$. We therefore have
\begin{align*}
  \left((O^*\Sigma_1 O)^{\frac 12} D(O^*\Sigma_1 O)^{\frac 12}   \right)^2 &=(O^*\Sigma_1 O)^{\frac 12}   P_1 O^*\Sigma_2 O P_1 (O^*\Sigma_1 O)^{\frac 12}
  \\
  &= O^* \Sigma_1^{\frac 12} \Sigma_2 \Sigma_1^{\frac 12}  O.
\end{align*}
This gives the first claim, and the first equality in the second claim easily follows. The  second equality is an immediate consequence of the definition of $D$.
\end{proof}

We now present a Lemma that will be useful to obtain the first order optimality condition. 
\begin{lemma}\label{lem_trace}
Let $\Sigma,\Sigma_1 \in \mathcal{S}_d^+$ be positive definite and $\Delta\in \cS_d$ be a symmetric matrix. Then, $f(t)=\tr((\Sigma^{1/2}(\Sigma_1+t\Delta)\Sigma^{1/2})^{1/2})$ and $g(t)=\BW^2(\Sigma,\Sigma_1+t\Delta)$ are $C^1$ on the open set $\{t\in\R:\Sigma_1+t\Delta\mbox{ is positive definite}\}$ which contains $0$ and we have 
\begin{align*}
  f'(t)&=\frac 12 \tr\left(\Sigma^{1/2}(\Sigma^{1/2}(\Sigma_1+t\Delta)\Sigma^{1/2})^{-1/2}\Sigma^{1/2}\Delta \right)\\
  g'(t)&=\tr\left(\left(I_d-\Sigma^{1/2}(\Sigma^{1/2}(\Sigma_1+t\Delta)\Sigma^{1/2})^{-1/2}\Sigma^{1/2}\right)\Delta \right).
\end{align*}
\end{lemma}
\begin{proof}
  Without loss of generality, it is sufficient to calculate the derivative at $t=0$. 
  We consider $\mathcal{L}:\mathcal{S}_d\to\mathcal{S}_d$ defined by  $\mathcal{L}(Y)=\tilde{\Sigma}_1^{1/2}Y+Y\tilde{\Sigma}_1^{1/2}$, with $\tilde{\Sigma}_1=\Sigma^{1/2}\Sigma_1\Sigma^{1/2}$ positive definite. Then, by~\cite[Lemma B.1]{AKR16}, $\mathcal{L}$ is invertible and $\delta=\mathcal{L}^{-1}(\Sigma^{1/2}\Delta \Sigma^{1/2})$ satisfies  $\delta+\tilde{\Sigma}_1^{-1/2}\delta\tilde{\Sigma}_1^{1/2}= \tilde{\Sigma}_1^{-1/2}\Sigma^{1/2}\Delta \Sigma^{1/2}$ so that $\tr(\delta)=\frac 1 2 \tr(\tilde{\Sigma}_1^{-{1/2}} \Sigma^{1/2}\Delta \Sigma^{1/2})$, by the cyclic property. By definition of $\delta$, we have 
  $$\left( (\tilde{\Sigma}_1)^{1/2} + t\delta\right)^2=\tilde{\Sigma}_1 +t \Sigma^{1/2}\Delta \Sigma^{1/2} +t^2 \delta^2.$$ 
  Since the square root matrix is $C^\infty$ on positive definite matrices (see e.g.~\cite[p.~134]{RoWi00}), we get that 
  $$(\Sigma^{1/2}(\Sigma_1+t\Delta)\Sigma^{1/2})^{1/2} =_{t\to 0^+} \tilde{\Sigma}_1^{1/2} +t \delta +O(t^2),$$
  and thus 
  $$\tr\left((\Sigma^{1/2}(\Sigma_1+t\Delta)\Sigma^{1/2})^{1/2}\right)=_{t\to 0^+} \tr\left((\Sigma^{1/2}\Sigma_1\Sigma^{1/2})^{1/2}\right)+t \frac 1 2 \tr(\tilde{\Sigma}_1^{-1/2} \Sigma^{1/2}\Delta \Sigma^{1/2}) +O(t^2),$$ 
  which gives $f'(0)=\frac 1 2 \tr(\tilde{\Sigma}_1^{-1/2} \Sigma^{1/2}\Delta \Sigma^{1/2})=\frac 1 2 \tr(\Sigma^{1/2} \tilde{\Sigma}_1^{-1/2} \Sigma^{1/2}\Delta )$. 
  
  Since $g(t)=\tr(\Sigma)+\tr(\Sigma_1+t\Delta)-2f(t)$, we then deduce $g'(0)$. 
\end{proof}
\begin{remark}
  Lemma 3.1 can also be deduced from more general differentiability result of the Wasserstein distance, either on the Wasserstein space of probability measures (see~\cite[Section 10.2]{AmGiSa08}) or on a lifted probability space, which we consider here. On an atomless probability space, we consider $Z\sim \mathcal{N}_d(0,I_d)$ and define $X_t=(\Sigma_1+t\Delta)^{1/2} Z \sim \mathcal{N}_d(0,\Sigma_1+t\Delta)$. The map $T(x)=\Sigma^{1/2}(\Sigma^{1/2}\Sigma_1\Sigma^{1/2})^{-1/2}\Sigma^{1/2} x$ is $\cW_2$-optimal between $\mathcal{N}_d(0,\Sigma_1)=\mathcal{L}(X_0)$ and  $\mathcal{N}_d(0,\Sigma)$. With~\cite[Theorem 3.2]{AlJo20} and the chain rule, we deduce that
  \begin{align*}
    g'(0)&=\E[2(I_d-\Sigma^{1/2}(\Sigma^{1/2}\Sigma_1\Sigma^{1/2})^{-1/2}\Sigma^{1/2})\Sigma_1^{1/2} Z \cdot {\bf D} Z ]\\
    &=\tr(2(I_d-\Sigma^{1/2}(\Sigma^{1/2}\Sigma_1\Sigma^{1/2})^{-1/2}\Sigma^{1/2})\Sigma_1^{1/2}{\bf D}),
  \end{align*}
  where ${\bf D}\in \cS_d$ is the derivative of $(\Sigma_1+t\Delta)^{1/2}$ at $t=0$ that thus satisfies ${\bf D} \Sigma_1^{1/2}+\Sigma_1^{1/2} {\bf D}= \Delta$. For any $M\in \cS_d$, we have $\tr(M\Sigma_1^{1/2} {\bf D})=\tr({\bf D} \Sigma_1^{1/2} M)$ by transposition and therefore  $\tr(M\Sigma_1^{1/2} {\bf D})=\frac 12 \tr(M \Delta )$. With the above equation, we finally get $g'(0)=\tr((I_d-\Sigma^{1/2}(\Sigma^{1/2}(\Sigma_1+t\Delta)\Sigma^{1/2})^{-1/2}\Sigma^{1/2})\Sigma_1^{1/2}\Delta)$.
  \end{remark}

 \subsection{Proof of Theorem~\ref{thmprojgauss}}
 
The proof of Theorem \ref{thmprojgauss} relies on the next proposition.
\begin{proposition}\label{prop_corot}
   Let $\Sigma_\mu,\Sigma_\nu\in\cS_d^+$ and $O\in\cO_d$ be such that $DO^*\Sigma_\mu OD\le O^*\Sigma_\nu O$ where $D={\rm diag}\left(1\wedge\sqrt{\frac{(O^*\Sigma_\nu O)_{11}}{(O^*\Sigma_\mu O)_{11}}},\cdots,1\wedge\sqrt{\frac{(O^*\Sigma_\nu O)_{dd}}{(O^*\Sigma_\mu O)_{dd}}}\right)$. Then, we have
\begin{align*}&\cI_2(\Sigma_\mu,\Sigma_\nu)=ODO^*\Sigma_\mu O DO^*,\\& \BW^2(\Sigma_\mu,\cI_2(\Sigma_\mu,\Sigma_\nu))=\sum_{i=1}^d\left(\sqrt{(O^*\Sigma_\mu O)_{ii}}-\sqrt{(O^*\Sigma_\nu O)_{ii}}\right)_+^2.\end{align*}
\end{proposition}
The proof of Proposition \ref{prop_corot} relies on the following lemma and is postponed after its proof.  
\begin{lemma}\label{lemrot}
  For $\Sigma_\mu,\Sigma_\nu\in\cS^+_d$ and $O\in\cO_d$, we have $$\cI_2(\Sigma_\mu,\Sigma_\nu)=O\cI_2(O^*\Sigma_\mu O,O^*\Sigma_\nu O)O^* \text{ and } \cJ_2(\Sigma_\mu,\Sigma_\nu)=O\cJ_2(O^*\Sigma_\mu O,O^*\Sigma_\nu O)O^*.$$
\end{lemma}
\begin{proof}[Proof of Lemma \ref{lemrot}]
  This directly follows from the equivalence $\Sigma\le\tilde\Sigma\iff O^*\Sigma O\le O^*\tilde \Sigma O$ and the equality $\BW^2(\Sigma,\tilde \Sigma)=\BW^2(O^*\Sigma O,O^*\tilde \Sigma O)$ that can  be deduced from~\eqref{w2gauss}, the cyclic property of the trace and the equality $(O^*\Sigma O)^{1/2}=O^*\Sigma^{1/2}O$.
\end{proof}
\begin{proof}[Proof of Proposition \ref{prop_corot}]
From Lemma~\ref{lemrot}, it is sufficient to prove the result for $O=I_d$, which we assume now.  

We set $\Sigma_{\cI}=D\Sigma_\mu D$. Let us first conclude when $\Sigma_\mu$, $\Sigma_\nu$ and therefore $\Sigma_{\cI}$ are positive definite before checking that we can get rid of this additional assumption. Since $\Sigma_\mu^{1/2}D\Sigma_\mu D\Sigma_\mu^{1/2}=\Sigma_\mu^{1/2}D\Sigma_\mu^{1/2}\Sigma_\mu^{1/2}D\Sigma_\mu^{1/2}$, $(\Sigma_\mu^{1/2}\Sigma_{\cI}\Sigma_\mu^{1/2})^{1/2}=\Sigma_\mu^{1/2}D\Sigma_\mu^{1/2}
 $.
By the cyclic property of the trace, \begin{align*}
                             \BW^2(\Sigma_\mu,\Sigma_\cI)&={\rm tr}\left(\Sigma_\mu+\Sigma_{\cI}-2D\Sigma_\mu\right)\\&=\sum_{i=1}^d\left((\Sigma_\mu)_{ii}+(\Sigma_\mu)_{ii}\wedge (\Sigma_\nu)_{ii}-2\sqrt{(\Sigma_\mu)_{ii}}\left(\sqrt{(\Sigma_\mu)_{ii}}\wedge \sqrt{(\Sigma_\nu)_{ii}}\right)\right)\\&=\sum_{i=1}^d\left(\sqrt{(\Sigma_\mu)_{ii}}-\sqrt{(\Sigma_\nu)_{ii}}\right)_+^2.
\end{align*}
Let $\Sigma\in\cS^+_d$ be such that $\Sigma\le\Sigma_\nu$ and $\Delta=\Sigma-\Sigma_{\cI}$. Since $\Sigma_\mu^{1/2}(\Sigma_\mu^{1/2}\Sigma_{\cI}\Sigma_\mu^{1/2})^{-1/2}\Sigma_\mu^{1/2}=D^{-1}$, by Lemma \ref{lem_trace},  
\begin{align}\frac{d}{dt}\BW^2(\Sigma_\mu,\Sigma_{\cI}+t\Delta)\big|_{t=0}&={\rm tr}((I_d-D^{-1})\Delta)\label{derw22}\\&={\rm tr}((D^{-1}-I_d)(\Sigma_{\cI}-\Sigma_\nu))+{\rm tr}\left((D^{-1}-I_d)(\Sigma_\nu-\Sigma)\right).\notag\end{align}
The definitions of $\Sigma_{\cI}$ and $D$ imply that for $i\in\{1,\cdots,d\}$, when $D_{ii}\ne 1$, then $(\Sigma_{\cI})_{ii}=(\Sigma_{\nu})_{ii}$. Therefore ${\rm tr}((D^{-1}-I_d)(\Sigma_{\cI}-\Sigma_\nu))=0$. Also using $D^{-1}-I_d\in\cS^+_d$, we deduce that 
\begin{align*}&\frac{d}{dt}\BW^2(\Sigma_\mu,\Sigma_{\cI}+t\Delta)\big|_{t=0}={\rm tr}\left((\Sigma_\nu-\Sigma)^{1/2}(D^{-1}-I_d)(\Sigma_\nu-\Sigma)^{1/2}\right)\ge 0.\end{align*}With the convexity of $\cS^+_d\ni\tilde\Sigma\mapsto \BW^2(\Sigma_\mu,\tilde\Sigma)$ stated in Lemma~\ref{lemconvw2}, we conclude that $\BW^2(\Sigma_\mu,\Sigma)\ge \BW^2(\Sigma_\mu,\Sigma_{\cI})$.

To get rid of the additional assumption that $\Sigma_\mu$ and $\Sigma_\nu$ are positive definite, we now construct for $\varepsilon>0$ positive definite matrices $\Sigma_\mu^{(\varepsilon)}$ and $\Sigma_\nu^{(\varepsilon)}$ converging to $\Sigma_\mu$ and $\Sigma_{\nu}$ as $\varepsilon\to 0$ and such that $D^{(\varepsilon)}\Sigma^{(\varepsilon)}_\mu D^{(\varepsilon)}\le \Sigma^{(\varepsilon)}_\nu$ for \begin{equation}
 D^{(\varepsilon)}={\rm diag}\left(1\wedge\sqrt{\frac{(\Sigma^{(\varepsilon)}_\nu)_{11}}{(\Sigma^{(\varepsilon)}_\mu)_{11}}},\cdots,1\wedge\sqrt{\frac{(\Sigma^{(\varepsilon)}_\nu)_{dd}}{(\Sigma^{(\varepsilon)}_\mu)_{dd}}}\right).\label{defdeps}  
\end{equation}
For $\varepsilon>0$, we define the positive definite matrix $\Sigma_\mu^{(\varepsilon)}=I^{(\varepsilon)}\Sigma_\mu I^{(\varepsilon)}+\varepsilon \Delta^{(\varepsilon)}$ where $I^{(\varepsilon)},\Delta^{(\varepsilon)}\in\cS_d^+$ are the diagonal matrices defined by \begin{align}\forall i\in\{1,\cdots,d\},\;I^{(\varepsilon)}_{ii}&=1_{\{(\Sigma_\nu)_{ii}=0\}}+1_{\{(\Sigma_\nu)_{ii}>0\}}D_{ii}\left(1\vee\sqrt{\frac{(\Sigma_\mu)_{ii}+\varepsilon}{(\Sigma_\nu)_{ii}+\varepsilon}}\right)\notag\\\Delta^{(\varepsilon)}_{ii}&=1_{\{(\Sigma_\nu)_{ii}=0\}}+1_{\{(\Sigma_\nu)_{ii}>0\}}\left(1\vee{\frac{(\Sigma_\mu)_{ii}+\varepsilon}{(\Sigma_\nu)_{ii}+\varepsilon}}\right)\label{defhatd}\end{align} and the matrix $\Sigma^{(\varepsilon)}_\nu$ with diagonal entries $(\Sigma^{(\varepsilon)}_\nu)_{ii}=(\Sigma_\nu)_{ii}+\varepsilon$ and extradiagonal entries ($i\not= j$)\begin{align*}
   (\Sigma^{(\varepsilon)}_\nu)_{ij}=1_{\{(\Sigma_\nu)_{ii}(\Sigma_\nu)_{jj}>0\}}(\Sigma_\nu)_{ij}&+1_{\{(\Sigma_\nu)_{ii}=0,(\Sigma_\nu)_{jj}>0\}}(\Sigma_\mu)_{ij}D_{jj}\sqrt{\frac{\varepsilon}{(\Sigma_\mu)_{ii}+\varepsilon}}\\&+1_{\{(\Sigma_\nu)_{ii}>0,(\Sigma_\nu)_{jj}=0\}}(\Sigma_\mu)_{ij}D_{ii}\sqrt{\frac{\varepsilon}{(\Sigma_\mu)_{jj}+\varepsilon}}\\&+1_{\{(\Sigma_\nu)_{ii}=0,(\Sigma_\nu)_{jj}=0\}}(\Sigma_\mu)_{ij}\sqrt{\frac{\varepsilon}{(\Sigma_\mu)_{ii}+\varepsilon}}\sqrt{\frac{\varepsilon}{(\Sigma_\mu)_{jj}+\varepsilon}}.
\end{align*} 
Since $(\Sigma_\mu)_{ij}=0$ when $(\Sigma_\mu)_{ii}(\Sigma_\mu)_{jj}=0$, $\Sigma^{(\varepsilon)}_\nu$ converges to  $\Sigma_\nu$  as $\varepsilon\to 0$. On the other hand, $\Delta^{(\varepsilon)}$ goes to $\Delta^{(0)}$ defined by replacing $\varepsilon$ by $0$ in the right-hand side of \eqref{defhatd} and $I^{(\varepsilon)}$ goes to $I_d$ so that $\Sigma_\mu^{(\varepsilon)}$ goes to $\Sigma_\mu$.  For $i\in\{1,\cdots,d\}$, if $(\Sigma_\mu)_{ii}\le(\Sigma_\nu)_{ii}$ or $(\Sigma_\nu)_{ii}=0$, then $I^{(\varepsilon)}_{ii}=1=\Delta^{(\varepsilon)}_{ii}$ and $(\Sigma_\mu^{(\varepsilon)})_{ii}=(\Sigma_\mu)_{ii}+\varepsilon$. When $0<(\Sigma_\nu)_{ii}<(\Sigma_\mu)_{ii}$, then $
I^{(\varepsilon)}_{ii}=\sqrt{\frac{(\Sigma_\nu)_{ii}\times ((\Sigma_\mu)_{ii}+\varepsilon)}{(\Sigma_\mu)_{ii}\times ((\Sigma_\nu)_{ii}+\varepsilon)}}$ and $\Delta^{(\varepsilon)}_{ii}=\frac{(\Sigma_\mu)_{ii}+\varepsilon}{(\Sigma_\nu)_{ii}+\varepsilon}$ so that 
$$(\Sigma^{(\varepsilon)}_\mu)_{ii}=\frac{(\Sigma_\nu)_{ii}\times ((\Sigma_\mu)_{ii}+\varepsilon)}{(\Sigma_\nu)_{ii}+\varepsilon}+\varepsilon\frac{(\Sigma_\mu)_{ii}+\varepsilon}{(\Sigma_\nu)_{ii}+\varepsilon}=(\Sigma_\mu)_{ii}+\varepsilon.$$
Therefore the matrix $D^{(\varepsilon)}$ defined in \eqref{defdeps} satisfies $D^{(\varepsilon)}={\rm diag}\left(1\wedge\sqrt{\frac{(\Sigma_\nu)_{11}+\varepsilon}{(\Sigma_\mu)_{11}+\varepsilon}},\cdots,1\wedge\sqrt{\frac{(\Sigma_\nu)_{dd}+\varepsilon}{(\Sigma_\mu)_{dd}+\varepsilon}}\right)$ and, for $i\in\{1,\cdots,d\}$, we have \begin{align*}
   (D^{(\varepsilon)}I^{(\varepsilon)})_{ii}&=1_{\{(\Sigma_\nu)_{ii}=0\}}\sqrt{\frac{\varepsilon}{(\Sigma_\mu)_{ii}+\varepsilon}}+1_{\{(\Sigma_\nu)_{ii}>0\}}D_{ii} 
                                              \notag\\(D^{(\varepsilon)}\Delta^{(\varepsilon)}D^{(\varepsilon)})_{ii}&=1_{\{(\Sigma_\nu)_{ii}=0\}}\frac{\varepsilon}{(\Sigma_\mu)_{ii}+\varepsilon}+1_{\{(\Sigma_\nu)_{ii}>0\}}.
                                                                                                                                                                                                                                                                                                                                           \end{align*}
 Moreover, $D^{(\varepsilon)}\Sigma^{(\varepsilon)}_\mu D^{(\varepsilon)}=(D^{(\varepsilon)}I^{(\varepsilon)})\Sigma_\mu(D^{(\varepsilon)}I^{(\varepsilon)})+\varepsilon(D^{(\varepsilon)}\Delta^{(\varepsilon)}D^{(\varepsilon)})$. 
For $i\in\{1,\cdots,d\}$ such that $(\Sigma_{\nu})_{ii}=0$, we thus have $$(D^{(\varepsilon)}\Sigma^{(\varepsilon)}_\mu D^{(\varepsilon)})_{ii}=(\Sigma^{(\varepsilon)}_\mu)_{ii}\frac{\varepsilon}{(\Sigma_\mu)_{ii}+\varepsilon} 
=\varepsilon=(\Sigma_\nu)_{ii}+\varepsilon=(\Sigma^{(\varepsilon)}_\nu)_{ii}.$$
For $i,j\in\{1,\cdots,d\}$ such that $i\ne j$ and $(\Sigma_{\nu})_{ii}(\Sigma_{\nu})_{jj}=0$, we have $(D^{(\varepsilon)}\Sigma^{(\varepsilon)}_\mu D^{(\varepsilon)})_{ij}=(\Sigma^{(\varepsilon)}_\nu)_{ij}$.
For $i,j\in\{1,\cdots,d\}$ such that $(\Sigma_{\nu})_{ii}(\Sigma_{\nu})_{jj}>0$, we have $(D^{(\varepsilon)}\Sigma^{(\varepsilon)}_\mu D^{(\varepsilon)})_{ij}=(D\Sigma_\mu D)_{ij}+\varepsilon 1_{\{i=j\}}$ while $(\Sigma^{(\varepsilon)}_\nu)_{ij}=(\Sigma_\nu)_{ij}+\varepsilon 1_{\{i=j\}}$. Therefore the inequality $D\Sigma_\mu D+\varepsilon I_d\le\Sigma_\nu+\varepsilon I_d$ used on the block corresponding to the indices $i\in\{1,\cdots,d\}$ such that $(\Sigma_\nu)_{ii}>0$ implies that $D^{(\varepsilon)}\Sigma^{(\varepsilon)}_\mu D^{(\varepsilon)}\le \Sigma^{(\varepsilon)}_\nu$, which in turn yields that $\Sigma^{(\varepsilon)}_\nu$ is positive definite.

Since $D^{(\varepsilon)}={\rm diag}\left(1\wedge\sqrt{\frac{(\Sigma_\nu)_{11}+\varepsilon}{(\Sigma_\mu)_{11}+\varepsilon}},\cdots,1\wedge\sqrt{\frac{(\Sigma_\nu)_{dd}+\varepsilon}{(\Sigma_\mu)_{dd}+\varepsilon}}\right)$, $D^{(\varepsilon)}$ converges to $D$ as $\varepsilon \to 0$. By taking the limit $\varepsilon \to 0$ in the equalities $$\cI_2(\Sigma^{(\varepsilon)}_\mu,\Sigma^{(\varepsilon)}_\nu)=D^{(\varepsilon)}\Sigma^{(\varepsilon)}_\mu D^{(\varepsilon)},\;\BW^2(\Sigma^{(\varepsilon)}_\mu,\cI_2(\Sigma^{(\varepsilon)}_\mu,\Sigma^{(\varepsilon)}_\nu))=\sum_{i=1}^d\left(\sqrt{(\Sigma^{(\varepsilon)}_\mu)_{ii}}-\sqrt{(\Sigma^{(\varepsilon)}_\nu)_{ii}}\right)_+^2, $$ and using the continuity of $\cI_2$ stated in Propositions \ref{lipmu} and \ref{holdnu}, we obtain $$\cI_2(\Sigma_\mu,\Sigma_\nu)=D\Sigma_\mu D\mbox{ and }\BW^2(\Sigma_\mu,\cI_2(\Sigma_\mu,\Sigma_\nu))=\sum_{i=1}^d\left(\sqrt{(\Sigma_\mu)_{ii}}-\sqrt{(\Sigma_\nu)_{ii}}\right)_+^2. $$
 \end{proof}

\begin{proposition}\label{prop_existsO}
  Let $\Sigma_\mu,\Sigma_\nu\in\cS_d^+$. Then, there exists $O\in\cO_d$ such that $DO^*\Sigma_\mu OD\le O^*\Sigma_\nu O$ where $D={\rm diag}\left(1\wedge\sqrt{\frac{(O^*\Sigma_\nu O)_{11}}{(O^*\Sigma_\mu O)_{11}}},\cdots,1\wedge\sqrt{\frac{(O^*\Sigma_\nu O)_{dd}}{(O^*\Sigma_\mu O)_{dd}}}\right)$.
\end{proposition}
\begin{proof}
  Let us first assume that $\Sigma_\mu,\Sigma_\nu\in\cS^{+,*}_d$. This implies $\Sigma_{\cJ}:=\cJ_2(\Sigma_\nu,\Sigma_\mu) \in \cS_d^{+,*}$ since $\Sigma_\mu\le \Sigma_{\cJ}$ . By Theorem~\ref{thm_BW}, there exist $O\in \cO_d$ and  $C\in \mathfrak{C}_d$ such that 
  $$O^*\Sigma_\nu O= \sqrt{\textup{dg}(O^*\Sigma_\nu O)}C \sqrt{\textup{dg}(O^*\Sigma_\nu O)},\  O^*\Sigma_\cJ O=\sqrt{\textup{dg}(O^*\Sigma_\cJ O)}C \sqrt{\textup{dg}(O^*\Sigma_\cJ O)}.$$
  This gives $O^*\Sigma_\nu O=DO^*\Sigma_\cJ OD$ with $D={\rm diag}\left(\sqrt{\frac{(O^*\Sigma_\nu O)_{11}}{(O^*\Sigma_\cJ O)_{11}}},\cdots,\sqrt{\frac{(O^*\Sigma_\nu O)_{dd}}{(O^*\Sigma_\cJ O)_{dd}}}\right)$. 
  We note that $D^{-1}=\textup{dg}(O^*\Sigma_\cJ O)^{1/2}\textup{dg}(O^*\Sigma_\nu O)^{-1/2}$ and thus 
  \begin{align}\label{eq_D}
    D^{-1}&=O^*\Sigma_\nu^{-1/2}(\Sigma_\nu^{1/2}\Sigma_\cJ \Sigma_\nu^{1/2})^{1/2}\Sigma_\nu^{-1/2}O,
  \end{align}
by Lemma~\ref{lem_calcD}.

Let $\Sigma \ge \Sigma_\mu$ and $\Sigma(\lambda)=(1-\lambda)\Sigma_\cJ+\lambda \Sigma$. We have by Lemma~\ref{lem_trace}
$$\frac{d}{d\lambda} \BW^2(\Sigma_\nu,\Sigma(\lambda))=\tr\left(\left(I_d-\Sigma_\nu^{1/2}(\Sigma_\nu^{1/2}\Sigma(\lambda)\Sigma_\nu^{1/2})^{-1/2}\Sigma_\nu^{1/2}\right)(\Sigma-\Sigma_\cJ)\right),$$
for all $\lambda \in [0,1]$, since $\Sigma(\lambda)\in \cS_d^{+,*}$. 
 Then using the optimality of $\Sigma_\cJ$, we get by taking this derivative at $\lambda=0$ and using~\eqref{eq_D}:
$$\tr((I_d-ODO^*)(\Sigma-\Sigma_\cJ))=\tr((I_d-D)(O^*\Sigma O-O^* \Sigma_\cJ O ))\ge 0.$$
By taking $\Sigma=\Sigma_\cJ+ O\Delta O^*$ with any diagonal matrix $\Delta \in \cS_d^+$, we get $\tr((I_d-D)\Delta)\ge 0$ and thus $D_{ii}\in [0,1]$ for all $i$. 

Besides, for $\Sigma=\Sigma_\mu$, we can take $\lambda\le 0$ since then $\Sigma(\lambda)=\Sigma_\cJ+\lambda(\Sigma_\mu-\Sigma_\cJ)\ge \Sigma_\cJ\ge \Sigma_\mu$, which gives $\tr((I_d-D)(O^*\Sigma_\mu O-O^* \Sigma_\cJ O))=0$. Therefore, $D_{ii}=1$ if $(O^*\Sigma_\cJ O)_{ii}>(O^*\Sigma_\mu O)_{ii}$, which gives that $(O^*\Sigma_\cJ O)_{ii}=(O^*\Sigma_\nu O)_{ii}$. Otherwise we have $(O^*\Sigma_\cJ O)_{ii}=(O^*\Sigma_\mu O)_{ii}\ge(O^*\Sigma_\nu O)_{ii}$ since $D_{ii}\le 1$. This gives $(O^*\Sigma_\cJ O)_{ii}=(O^*\Sigma_\nu O)_{ii}\vee (O^*\Sigma_\mu O)_{ii}$ for $i\in\{1,\cdots,d\}$ and thus
$$D={\rm diag}\left(1\wedge\sqrt{\frac{(O^*\Sigma_\nu O)_{11}}{(O^*\Sigma_\mu O)_{11}}},\cdots,1\wedge\sqrt{\frac{(O^*\Sigma_\nu O)_{dd}}{(O^*\Sigma_\mu O)_{dd}}}\right).$$
We conclude by observing that $O^*\Sigma_\nu O=DO^*\Sigma_\cJ OD\ge DO^*\Sigma_\mu OD$.

Last, we now consider the case $\Sigma_\mu,\Sigma_\nu \in \cS_d^+$ and define, for $n\ge 1$, $\Sigma_\mu^n = \Sigma_\mu + \frac 1n I_d$, $\Sigma_\nu^n = \Sigma_\nu + \frac 1n I_d$. We note $O_n \in \mathcal{O}_d$ and $C_n\in \mathfrak{C}_d$ such that $O^*_n\Sigma_\nu^n O $ and $O^*_n\mathcal{J}_2(\Sigma_\nu^n, \Sigma_\mu^n) O_n$ share the same correlation matrix $C_n$.  We have $D_n O^*_n \Sigma_\mu^n O_n D_n\le   O^*_n \Sigma_\nu^n O_n $ with 
$$D_n={\rm diag}\left(1\wedge\sqrt{\frac{(O^*_n\Sigma^n_\nu O_n)_{11}}{(O_n^*\Sigma^n_\mu O_n)_{11}}},\cdots,1\wedge\sqrt{\frac{(O_n^*\Sigma^n_\nu O_n)_{dd}}{(O_n^*\Sigma^n_\mu O_n)_{dd}}}\right).$$
By compacity of ${\mathcal O}_d\times {\mathfrak C}_d$ and the set of diagonal matrices with coefficients in $[0,1]$, we can extract a subsequence of $(O_n,C_n,D_n)$ that converges to $(O,C,\tilde{D})$ and we have $\tilde{D}O^*\Sigma_\mu O \tilde{D} \le O^* \Sigma_\nu O$. For the matrix $D$ in the statement, we clearly have $\tilde{D}_{ii}=D_{ii}$ if $(O^*\Sigma_\mu O)_{ii}>0$. We also note that $\tilde{D}O^*\Sigma_\mu O \tilde{D}$ does not depend on the value of $\tilde{D}_{ii}$ if $(O^*\Sigma_\mu O)_{ii}=0$ which implies $(O^*\Sigma_\mu O)_{ij}=(O^*\Sigma_\mu O)_{ji}=0$ for all $j\in\{1,\cdots,d\}$, so that we also have $DO^*\Sigma_\mu O D \le O^* \Sigma_\nu O$.
\end{proof}

 \begin{proof}[Proof of Theorem \ref{thmprojgauss}]
From Proposition~\ref{prop_existsO}, there exists 
$O\in\cO_d$ such that $DO^*\Sigma_\mu OD\le O^*\Sigma_\nu O$ where $D={\rm diag}\left(1\wedge\sqrt{\frac{(O^*\Sigma_\nu O)_{11}}{(O^*\Sigma_\mu O)_{11}}},\cdots,1\wedge\sqrt{\frac{(O^*\Sigma_\nu O)_{dd}}{(O^*\Sigma_\mu O)_{dd}}}\right)$. Then, by Proposition~\ref{prop_corot}, we get $\cI_2(\Sigma_\mu,\Sigma_\nu)=ODO^*\Sigma_\mu ODO^*$ and the formula for $\BW^2(\Sigma_\mu,\cI_2(\Sigma_\mu,\Sigma_\nu))$ holds by Theorem~\ref{thm_BW}.
Let now $\tilde\Sigma_{\cJ}$ be defined by \eqref{deftsigj}. By left and right-multiplications by $\tilde D={\rm diag}\left(\frac{1_{\{(O^*\Sigma_\nu O)_{11}>0\}}}{D_{11}},\cdots,\frac{1_{\{(O^*\Sigma_\nu O)_{dd}>0\}}}{D_{dd}}\right)$, the inequality $DO^*\Sigma_\mu OD\le O^*\Sigma_\nu O$ implies that \begin{align}
   \left(1_{\{(O^*\Sigma_\nu O)_{ii}(O^*\Sigma_\nu O)_{jj}>0\}}(O^*\Sigma_\mu O)_{ij}\right)_{1\le i,j\le d}&\le \tilde D O^*\Sigma_\nu O \tilde D=\left(1_{\{(O^*\Sigma_\nu O)_{ii}(O^*\Sigma_\nu O)_{jj}>0\}}(\tilde\Sigma_\cJ)_{ij}\right)_{1\le i,j\le d}.\label{egalnonnul}
\end{align}
 Since for $1\le i,j\le d$,  $(\tilde\Sigma_{\cJ})_{ij}=(O^*\Sigma_\mu O)_{ij}$ when  $(O^*\Sigma_\nu O)_{ii}(O^*\Sigma_\nu O)_{jj}=0$, we deduce that $O^*\Sigma_\mu O\le\tilde\Sigma_{\cJ}$ and therefore $\Sigma_\mu\le O\tilde\Sigma_{\cJ}O^*$. In view of \eqref{fb}, to deduce that $\cJ_2(\Sigma_\nu,\Sigma_\mu)=O\tilde\Sigma_\cJ O^*$, it is enough to check that \begin{equation}
   \BW^2(\Sigma_\nu,O\tilde\Sigma_\cJ O^*)=\sum_{i=1}^d\left(\sqrt{(O^*\Sigma_\mu O)_{ii}}-\sqrt{(O^*\Sigma_\nu O)_{ii}}\right)_+^2.\label{w22nj}
 \end{equation} 
 In fact, let $C \in \mathfrak{C}_d$ be a correlation matrix of $\tilde{\Sigma}_\cJ$, i.e. $\tilde{\Sigma}_\cJ=\textup{dg}(\tilde{\Sigma}_\cJ)^{1/2} C \textup{dg}(\tilde{\Sigma}_\cJ)^{1/2}$. Then, $C$ is also a correlation matrix of $O^*\Sigma_\nu O=D \tilde{\Sigma}_\cJ D$, and we get~\eqref{w22nj} by Theorem~\ref{thm_BW}. 
\end{proof}

\begin{remark}
  We note that for the matrix $O\in \mathcal{O}_d$ given by Proposition~\ref{prop_existsO}, there exists $C,C'\in \mathfrak{C}_d$ such that $O^*\Sigma_\nu O$ and $O^* \cJ_2(\Sigma_\nu,\Sigma_\mu)O=\tilde{\Sigma}_\cJ$ share the correlation matrix $C$, and $O^*\Sigma_\mu O$ and $O^* \cI_2(\Sigma_\mu,\Sigma_\nu)O$ share the correlation matrix $C'$.
\end{remark}

\subsection{Further results on the projections}

\begin{corollary}\label{cormgn} Let $\Sigma_\mu,\Sigma_\nu\in\cS^+_d$. Then,
 \begin{align*}
   \cI_2(\Sigma_\mu,\Sigma_\nu)=\Sigma_\nu&\iff\cJ_2(\Sigma_\nu,\Sigma_\mu)=\Sigma_\mu \\ &\iff \exists C\in \mathfrak{C}_d,O\in \mathcal{O}_d
   \text{ such that }   O^*\Sigma_\mu O \text { and }O^* \Sigma_\nu O \text{ share the}\\& \text{ correlation matrix }C \text{ and } \ \textup{dg}(O^*\Sigma_\nu O)\le\textup{dg}(O^*\Sigma_\mu O) 
 \\
&\iff \Sigma_\nu\le (\Sigma_\nu^{1/2}\Sigma_\mu\Sigma_\nu^{1/2})^{1/2}\\
     &\;\;\Longrightarrow (\Sigma^{1/2}_\mu\Sigma_\nu\Sigma_\mu^{1/2})^{1/2}\le\Sigma_\mu,
 \end{align*}
where the last implication is an equivalence when $\Sigma_\mu$ is positive definite but not in general.
 \end{corollary}
 \begin{remark}\label{remequivinegsqrt}\begin{itemize}
   \item When $\Sigma_\nu\le\Sigma_\mu$, then $\Sigma_\nu^2\le\Sigma_\nu^{1/2}\Sigma_\mu\Sigma_\nu^{1/2}$ and, by Lemma \ref{ordsq}, $\Sigma_\nu\le(\Sigma_\nu^{1/2}\Sigma_\mu\Sigma_\nu^{1/2})^{1/2}$. With Corollary \ref{cormgn}, we deduce that $\Sigma_\nu\le\Sigma_\mu$ implies that $\cI_2(\Sigma_\mu,\Sigma_\nu)=\Sigma_\nu$ and $\cJ_2(\Sigma_\nu,\Sigma_\mu)=\Sigma_\mu$. This is particular to the Gaussian case.
     Although perhaps counterintuitive, the inequality $\nu\le_{cx}\mu$ is in general not enough to ensure that $\cI_2(\mu,\nu)=\nu$ and $\cJ_2(\nu,\mu)=\mu$. Indeed, when $d=1$, by \eqref{eq:quantileProjections}, $\cI_2(\mu,\nu)=\nu$ and $\cJ_2(\nu,\mu)=\mu$ if and only if the the function $G(u)=\int_0^u(F_\mu^{-1}-F_\nu^{-1})(v)dv$ is convex on $[0,1]$ i.e. $F_\mu^{-1}-F_\nu^{-1}$ is non-decreasing on $(0,1)$, while $\nu\le_{cx}\mu$ is equivalent to the non-positivity of $G$ together with $G(1)=0$.
     \item Let us note that the lack of equivalence between $ \Sigma_\nu\le (\Sigma_\nu^{1/2}\Sigma_\mu\Sigma_\nu^{1/2})^{1/2}$ and $(\Sigma^{1/2}_\mu\Sigma_\nu\Sigma_\mu^{1/2})^{1/2} \le \Sigma_\mu$ already holds in dimension $d=1$. Indeed, for $\Sigma_\mu,\Sigma_\nu \in \R_+$, we have
$$ \Sigma_\nu \le \sqrt{\Sigma_\nu \Sigma_\mu } \iff \Sigma_\nu \le\Sigma_\mu  \implies  \sqrt{\Sigma_\nu  \Sigma_\mu } \le \Sigma_\mu,$$
where the last implication is an equivalence for $\Sigma_\mu>0$ but not when $0=\Sigma_\mu<\Sigma_\nu$. 
 \end{itemize}
   \end{remark}
\begin{proof}[Proof of Corollary \ref{cormgn}] 
  First, the equivalence
  $$\cI_2(\Sigma_\mu,\Sigma_\nu)=\Sigma_\nu \iff \cJ_2(\Sigma_\nu,\Sigma_\mu)=\Sigma_\mu$$ is a consequence of~\eqref{fb} and the  uniqueness of the projections given by Proposition~\ref{propj2}.
  
  If $\cJ_2(\Sigma_\nu,\Sigma_\mu)=\Sigma_\mu$, then Theorem~\ref{thmprojgauss} gives the existence of $O\in \mathcal{O}_d$ such that the matrix $\tilde\Sigma_\cJ$ defined by~\eqref{deftsigj} is equal to $O^*\Sigma_\mu O$. Besides, we have $DO^*\Sigma_\mu OD=O^*\Sigma_\nu O$, which gives that $O^*\Sigma_\mu O$ and  $O^*\Sigma_\nu O$ share the same correlation matrix and $\textup{dg}(O^*\Sigma_\nu O)\le\textup{dg}(O^*\Sigma_\mu O)$, because the entries of~$D$ are in $[0,1]$. Conversely, if there exists $C\in \mathfrak{C}_d$ and $O\in \mathcal{O}_d$  such that $O^*\Sigma_\mu O$ and  $O^*\Sigma_\nu O$ share the correlation matrix~$C$ and $\textup{dg}(O^*\Sigma_\nu O)\le\textup{dg}(O^*\Sigma_\mu O)$, then we have $DO^*\Sigma_\mu O D=O^*\Sigma_\nu O$ with $$D={\rm diag}\left(1\wedge\sqrt{\frac{(O^*\Sigma_\nu O)_{11}}{(O^*\Sigma_\mu O)_{11}}},\cdots,1\wedge\sqrt{\frac{(O^*\Sigma_\nu O)_{dd}}{(O^*\Sigma_\mu O)_{dd}}}\right),$$
  and we get $\cJ_2(\Sigma_\nu,\Sigma_\mu)=\Sigma_\mu$ by Theorem~\ref{thmprojgauss}.

  Suppose now that there exists $C\in\mathfrak{C}_d$ and $O\in \mathcal{O}_d$ such that $O^*\Sigma_\mu O$ and  $O^*\Sigma_\nu O$ share the correlation matrix~$C$ and $\textup{dg}(O^*\Sigma_\nu O)\le\textup{dg}(O^*\Sigma_\mu O)$. Then, we get that $\Sigma_\nu\le (\Sigma_\nu^{1/2}\Sigma_\mu\Sigma_\nu^{1/2})^{1/2}$ (resp. $(\Sigma^{1/2}_\mu\Sigma_\nu\Sigma_\mu^{1/2})^{1/2}\le\Sigma_\mu$)  by taking $\Sigma_1=\Sigma_\nu$ and $\Sigma_2=\Sigma_\mu$  (resp. $\Sigma_1=\Sigma_\mu$ and $\Sigma_2=\Sigma_\nu$) in Lemma~\ref{lem_calcD}  and multiplying to the left and to the right by $(O^*\Sigma_1O)^{1/2}$ the inequality $I_d\le D$ (resp. $D\le I_d$) satisfied by the diagonal matrix~$D$ defined therein and using \eqref{egqd}. The reciprocal is straightforward when $(\Sigma^{1/2}_\mu\Sigma_\nu\Sigma_\mu^{1/2})^{1/2}\le\Sigma_\mu$ and $\Sigma_\mu$ is invertible by using \eqref{egqd2} in Lemma~\ref{lem_calcD}. We now prove it in the general case and  assume $ \Sigma_\nu\le (\Sigma_\nu^{1/2}\Sigma_\mu\Sigma_\nu^{1/2})^{1/2}$.  
  From Theorem~\ref{thm_BW} we consider $C\in \mathfrak{C}_d$ and $O\in \mathcal{O}_d$  such that $O^*\Sigma_\mu O$ and  $O^*\Sigma_\nu O$ share the correlation matrix~$C$. We can do this in such a way that $\#\{i: (O^*\Sigma_\nu O)_{ii}>0\}=\textup{rk}(\Sigma_\nu)$.  Indeed, let $O_1 \in  \mathcal{O}_d$ be such that $O_1^*\Sigma_\nu O_1$ is diagonal with $(O_1^*\Sigma_\nu O_1)_{ii}>0$ for $i\le \textup{rk}(\Sigma_\nu)=:d'$. Then, by Theorem~\ref{thm_BW}, there exists $O_2  \in \mathcal{O}_{d'}$ such that $O_2^*(O_1^*\Sigma_\nu O_1)_{1:d',1:d'}O_2$ and $O_2^*(O_1^*\Sigma_\mu O_1)_{1:d',1:d'}O_2$ share the same correlation matrix. Let $O=O_1 \tilde{O}_2$, where $\tilde{O}_2\in\R^{d\times d}$ is the block diagonal matrix with diagonal blocks $O_2$ and $I_{d-d'}$. Then $O^*\Sigma_\mu O$ and  $O^*\Sigma_\nu O$ share the same correlation matrix. Then, by Lemma~\ref{lem_calcD}, we get 
  $$(O^*\Sigma_\nu O)^{1/2}\tilde{D}(O^*\Sigma_\nu O)^{1/2}=O^*(\Sigma_\nu^{1/2}\Sigma_\mu\Sigma_\nu^{1/2})^{1/2}O \ge  O^*\Sigma_\nu O,$$
with $\tilde{D}_{ii}=\sqrt{\frac{(O^*\Sigma_\mu O)_{ii}}{(O^*\Sigma_\nu O)_{ii}}}$ if $i\le d'$ and $\tilde{D}_{ii}=1$ otherwise. By construction, $O^*\Sigma_\nu O$ is a block diagonal matrix with the first block in $\cS_{d'}^{+,*}$, which gives $\tilde{D}_{ii}\ge 1$ for $i\le d'$, and therefore $\textup{dg}(O^*\Sigma_\nu O)\le \textup{dg}(O^*\Sigma_\mu O)$.
\end{proof}

\begin{remark}
  \begin{itemize}
    \item From the proof of Corollary~\ref{cormgn}, we see that when the equivalences hold, we have in fact $\textup{dg}(O^*\Sigma_\nu O)\le\textup{dg}(O^*\Sigma_\mu O)$ for any  $ C\in \mathfrak{C}_d$ and $O\in \mathcal{O}_d$ such that $O^*\Sigma_\mu O$ and $O^* \Sigma_\nu O$ share the correlation matrix $C$ and $\#\{i: (O^*\Sigma_\nu O)_{ii}>0\}=\textup{rk}(\Sigma_\nu)$.
    In particular, when $\Sigma_\nu$ is invertible, $\textup{dg}(O^*\Sigma_\nu O)\le\textup{dg}(O^*\Sigma_\mu O)$  for any  $ C\in \mathfrak{C}_d$ and $O\in \mathcal{O}_d$ such that $O^*\Sigma_\mu O$ and $O^* \Sigma_\nu O$ share the correlation matrix $C$.
  \item Let us point that the equivalence $\cI_2(\Sigma_\mu,\Sigma_\nu)=\Sigma_\nu \iff \exists C\in \mathfrak{C}_d,O\in \mathcal{O}_d$
    such that $O^*\Sigma_\mu O$ and $O^* \Sigma_\nu O$  share the  correlation matrix $C$  and $\textup{dg}(O^*\Sigma_\nu O)\le\textup{dg}(O^*\Sigma_\mu O)$  can be seen as a particular case of  \cite[Theorem 7.5]{KimRuan}. Indeed, if $X\sim \mathcal{N}_d(0,\Sigma_\mu)$ and $Y=ODO^*X\sim \mathcal{N}_d(0,\Sigma_\nu)$ with $D_{ii}=\sqrt{\frac{(O^*\Sigma_\nu O)_{ii}}{(O^*\Sigma_\mu O)_{ii}}}$  if $(O^*\Sigma_\mu O)_{ii}>0$, $D_{ii}=1$ otherwise, then $X$ and $Y$ are optimally coupled. The transport map $x\mapsto ODO^*x$ is the gradient of $\frac 12 x^* ODO^*x$ and is a convex contraction in the sense of~\cite[Definition 5.1]{KimRuan} since $ODO^*\le I_d$, using~\cite[Remark 5.3]{KimRuan}. In the same way, the equivalence $\cJ_2(\Sigma_\nu,\Sigma_\mu)=\Sigma_\nu \iff \exists C\in \mathfrak{C}_d,O\in \mathcal{O}_d$
    such that $O^*\Sigma_\mu O$ and $O^* \Sigma_\nu O$  share the  correlation matrix $C$  and $\textup{dg}(O^*\Sigma_\nu O)\le\textup{dg}(O^*\Sigma_\mu O)$ is also a consequence of~\cite[Theorem 7.7]{KimRuan}.
  \end{itemize}
  
\end{remark}

In the setting of Corollary~\ref{cormgn}, there exists $O \in \mathcal{O}_d$ such that $O^*\Sigma_\mu O$, $O^*\Sigma_\nu O$, $O^* \cI_2(\Sigma_\mu,\Sigma_\nu) O$ and $O^* \cJ_2(\Sigma_\mu,\Sigma_\nu) O$ share the same correlation matrix. By Remark~\ref{rk_thm_main}, this also  holds when $\Sigma_\mu \Sigma_\nu=\Sigma_\nu \Sigma_\mu$, in which case we can find $O\in \mathcal{O}_d$ such that all these matrices are diagonal and therefore share the identity correlation matrix. The next corollary embeds these particular cases, and gives a sufficient condition to have a common correlation matrix. Notably, the projections then are easy to calculate. 

\begin{corollary}\label{cor_DDhat}
  Let $O \in \mathcal{O}_d$ and $C\in \mathfrak{C}_d$ be such that $O^*\Sigma_\mu O$ and  $O^*\Sigma_\nu O$ share the correlation matrix~$C$. Let us assume that $(O^*\Sigma_\nu O)_{ii}>0$ and $(O^*\Sigma_\nu O)_{ii}>0$ for all $i$ and let define 
  \begin{align*}
    D&= \textup{diag}\left(1\wedge \sqrt{\frac{(O^*\Sigma_\nu O)_{11}}{(O^*\Sigma_\mu O)_{11}}},\dots,1\wedge \sqrt{\frac{(O^*\Sigma_\nu O)_{dd}}{(O^*\Sigma_\mu O)_{dd}}}\right)\\
    \hat{D}&=\textup{diag}\left(1\wedge \sqrt{\frac{(O^*\Sigma_\mu O)_{11}}{(O^*\Sigma_\nu O)_{11}}},\dots,1\wedge \sqrt{\frac{(O^*\Sigma_\mu O)_{dd}}{(O^*\Sigma_\nu O)_{dd}}}\right).
  \end{align*}
If $\hat{D} C \hat{D}\le C$, then we have $\cI_2(\Sigma_\mu,\Sigma_\nu)=ODO^*\Sigma_\mu O DO^*$ and $\cJ_2(\Sigma_\nu,\Sigma_\mu)=OD^{-1}O^*\Sigma_\nu O D^{-1}O^*$.
When $d=2$, $$\hat{D} C \hat{D}\le C \iff \hat{D}=I_2\mbox{ or } C_{12}^2\le 1-\frac{(\hat{D}_{11}-\hat{D}_{22})^2}{(1-\hat{D}_{11}\hat{D}_{22})^2},$$
and for $d\ge 3$, a necessary condition is to have $C_{ij}^2\left(1-\hat{D}_{ii}\hat{D}_{jj} \right)^2\le \left(1-\hat{D}_{ii}^2\right)\left(1-\hat{D}_{jj}^2\right)$, for all $1\le i<j\le d$.
\end{corollary}
Before proving this result, we note that the condition $\hat{D} C \hat{D}\le C$ is trivially satisfied when $C=I_d$. Heuristically, we may hope that it is still satisfied if the correlations are sufficiently small. By contrast, if $C$ is such that $C_{ij}=1$ for all $1\le i \le j\le d$, then the necessary condition gives $\hat{D}_{ii}^2+\hat{D}_{jj}^2 \le 2\hat{D}_{ii}\hat{D}_{jj}$, and thus $\hat{D}=\sqrt{c} I_d$ with $c\in (0,1]$. This occurs either if $\Sigma_\mu=c\Sigma_\nu$ or $(O^*\Sigma_\mu O)_{ii}\ge(O^*\Sigma_\nu O)_{ii}$ for all $i$. The latter condition also appears in Corollary~\ref{cormgn}, and thus the condition $\hat{D} C \hat{D}\le C$ implies either $\cI_2(\Sigma_\mu,\Sigma_\nu)=\Sigma_\mu$ or $\cI_2(\Sigma_\mu,\Sigma_\nu)=\Sigma_\nu$ for this particular correlation matrix.

\begin{proof}
  Let us recall that the existence of $O$ and $C$ is given by Theorem~\ref{thm_BW}. Thus, $O^*\Sigma_\mu O=\textup{dg}(O^*\Sigma_\mu O)^{1/2} C\textup{dg}(O^*\Sigma_\mu O)^{1/2}$ and $O^*\Sigma_\nu O=\textup{dg}(O^*\Sigma_\nu O)^{1/2} C\textup{dg}(O^*\Sigma_\nu O)^{1/2}$. We then get 
  \begin{align*}
    D O^*\Sigma_\mu O D=\textup{dg}(O^*\Sigma_\nu O)^{1/2} \hat{D} C \hat{D} \textup{dg}(O^*\Sigma_\nu O)^{1/2} \le O^*\Sigma_\nu O,
  \end{align*}
by using that $\hat{D} C \hat{D}\le C$. Theorem~\ref{thmprojgauss} gives the formulas for the projections.

In dimension~$d=2$, since the diagonal coefficients of $\hat{D} C \hat{D}$ are not greater than those of $C$, we have $\hat{D} C \hat{D}\le C \iff \det(C-\hat{D} C \hat{D})\ge 0$. In higher dimension, any minor of size~$2$ has to be non-negative, which gives the necessary condition.
\end{proof}

Corollary~\ref{cor_DDhat} gives a partial way to calculate the projections $\cI_2(\Sigma_\mu,\Sigma_\nu)$ and $\cJ_2(\Sigma_\mu,\Sigma_\nu)$. We first calculate $O \in \mathcal{O}_d$ and $C\in \mathfrak{C}_d$ such that $O^*\Sigma_\mu O$ and  $O^*\Sigma_\nu O$ share the correlation matrix~$C$. This can be done by diagonalising the matrix $\Sigma_\mu^{-1/2}(\Sigma_\mu^{1/2} \Sigma_\nu \Sigma_\mu^{1/2})^{1/2}\Sigma_\mu^{-1/2}$ when $\Sigma_\mu\in \mathcal{S}_d^{+,*}$ or the matrix $\Sigma_\nu^{-1/2}(\Sigma_\nu^{1/2} \Sigma_\mu \Sigma_\nu^{1/2})^{1/2}\Sigma_\nu^{-1/2}$ when $\Sigma_\nu\in \mathcal{S}_d^{+,*}$. Then, if $\hat{D} C \hat{D}\le C$, the projections $\cI_2(\Sigma_\mu,\Sigma_\nu)$ and $\cJ_2(\Sigma_\mu,\Sigma_\nu)$ can be easily calculated. Unfortunately, this condition may not hold, and we present in the next subsection a general method to calculate these projections. 

\subsection{A projected gradient algorithm to calculate $\cI_2(\Sigma_\mu,\Sigma_\nu)$ and $\cJ_2(\Sigma_\mu,\Sigma_\nu)$}

Since, by Lemma~\ref{lemconvw2}, the map $\mathcal{S}_d^+\ni \Sigma \mapsto \BW^2(\Sigma_\mu,\Sigma)$ (resp. $\mathcal{S}_d^+\ni \Sigma \mapsto \BW^2(\Sigma_\nu,\Sigma)$) is convex, the numerical calculation of $\cI_2(\Sigma_\mu,\Sigma_\nu)$ (resp. $\cJ_2(\Sigma_\mu,\Sigma_\nu)$) can be performed by a projected gradient descent on the  closed convex set $\{ \Sigma \in  \cS_d^+ : \Sigma \le \Sigma_\nu \}$ (resp. $\{ \Sigma \in  \cS_d^+ : \Sigma \ge \Sigma_\mu \}$). The gradient of the above mapping is given by Lemma~\ref{lem_trace} : for $\Sigma,\Sigma_\mu \in \mathcal{S}_d^{+,*}$, $\nabla_{\Sigma}\BW^2(\Sigma_\mu,\Sigma)=I_d-\Sigma_\mu^{1/2}(\Sigma^{1/2}_\mu\Sigma\Sigma^{1/2}_\mu)^{-1/2}\Sigma^{1/2}_\mu$.  The projections on the closed convex sets $\{ \Sigma \in  \cS_d : \Sigma \le \Sigma_\nu \}$ (resp. $\{ \Sigma \in  \cS_d^+ : \Sigma \ge \Sigma_\mu \}$) for the Euclidean norm, also known as the Frobenius norm, $\R^{d\times d}\ni A\mapsto \sqrt{\tr(AA^*)}$ on the set of square matrices is given by the next proposition. Note that the projection on $\{ \Sigma \in  \cS_d^+ : \Sigma \le \Sigma_\nu \}$ is, up to our knowledge, not explicit. 

\begin{proposition}
  Let $\Sigma_\mu, \Sigma_\nu \in \cS_d^+$.  The Frobenius projection of $\Sigma_\mu$ (resp. $\Sigma_\nu$) on the closed convex set $\{ \Sigma \in  \cS_d : \Sigma \le \Sigma_\nu \}$ (resp. $\{ \Sigma \in  \cS_d^+ : \Sigma \ge \Sigma_\mu \}$) is given by 
  $$ \Sigma_\mu-(\Sigma_\mu-\Sigma_\nu)^+ \ (\text{resp. } \Sigma_\nu + (\Sigma_\mu-\Sigma_\nu)^+ ),$$
  where, for a symmetric matrix $M=ODO^*$ with $O\in\cO_d$ and $D$ diagonal, $M^+=OD^+O^*$ with $D^+$  being the diagonal matrix such that $(D^+)_{ii}=\max(D_{ii},0)$. 
\end{proposition}
\begin{proof}
  We first notice that  $0\le(\Sigma_\mu-\Sigma_\nu)^++\Sigma_\nu-\Sigma_\mu$ and thus $\Sigma_\mu-(\Sigma_\mu-\Sigma_\nu)^+\le \Sigma_\nu$. 
  Let $\Sigma \in  \cS_d^+$ such that $\Sigma \le \Sigma_\nu$.
  We have $\Sigma-\Sigma_\mu= \left( \Sigma-\Sigma_\mu+(\Sigma_\mu-\Sigma_\nu)^+\right)-(\Sigma_\mu-\Sigma_\nu)^+$ and thus 
\begin{align*}
  {\rm tr}\left( (\Sigma-\Sigma_\mu)^2\right)=&{\rm tr}\left(\left( \Sigma-\Sigma_\mu+(\Sigma_\mu-\Sigma_\nu)^+\right)^2\right)+{\rm tr}\left( ((\Sigma_\mu-\Sigma_\nu)^+)^2\right)\\
  & -2{\rm tr}\left(\left( \Sigma-\Sigma_\nu+\Sigma_\nu-\Sigma_\mu+(\Sigma_\mu-\Sigma_\nu)^+\right) (\Sigma_\mu-\Sigma_\nu)^+\right) \\\ge&  {\rm tr}\left(\left( \Sigma-\Sigma_\mu+(\Sigma_\mu-\Sigma_\nu)^+\right)^2\right)+{\rm tr}\left( ((\Sigma_\mu-\Sigma_\nu)^+)^2\right),
\end{align*}
since ${\rm tr}((\Sigma-\Sigma_\nu)(\Sigma_\mu-\Sigma_\nu)^+)\le 0$ and  $\left( \Sigma_\nu-\Sigma_\mu+(\Sigma_\mu-\Sigma_\nu)^+\right) (\Sigma_\mu-\Sigma_\nu)^+=0$. This shows the optimality of $\Sigma_\mu-(\Sigma_\mu-\Sigma_\nu)^+$, and the proof is similar for the other projection.
\end{proof}

It is important to note that the Frobenius projection on $\{ \Sigma \in  \cS_d : \Sigma \le \Sigma_\nu \}$ may not be positive semidefinite because the positive part may not preserve the matrix order, see for instance~\cite[Remark 7 (3)]{joupagm}. Thus, it more convenient to work with the other projection and to calculate first $\cJ_2(\Sigma_\nu,\Sigma_\mu)$. % The gradient of $\BW^2$ is given by Lemma~\ref{lem_trace}.
For $\Sigma_\mu,\Sigma_\nu \in \cS_d^{+,*}$, the projected gradient descent for the calculation of   $\cJ_2(\Sigma_\nu,\Sigma_\mu)$ can be written as follows. 
Let $(\eta_i)_{i\ge 1}$ be a nonincreasing sequence of positive numbers. Initialize $\Sigma=\Sigma_\nu+(\Sigma_\mu-\Sigma_\nu)^+$, $i=1$ and repeat, until some stopping criterion, the following steps
\begin{itemize}
\item $\Sigma \leftarrow \Sigma - \eta_i \left(I_d-\Sigma_\nu^{1/2}(\Sigma_\nu^{1/2} \Sigma \Sigma_\nu^{1/2})^{-1/2}\Sigma_\nu^{1/2}\right)$,
\item $\Sigma \leftarrow \Sigma+(\Sigma_\mu-\Sigma)^+$,
\item $i\leftarrow i+1$. 
\end{itemize}
In our numerical experiments in small dimension, a constant time-step $\eta_i$ gradient descent converges rather quickly.

Once $ \cJ_2(\Sigma_\nu,\Sigma_\mu)$ is known, we can compute $\cI_2(\Sigma_\mu,\Sigma_\nu)$ by using Theorem~\ref{thmprojgauss} as follows. We compute $O\in{\mathcal O}_d$ such that $O^*\Sigma_\nu^{-1/2}(\Sigma_\nu^{1/2} \cJ_2(\Sigma_\nu,\Sigma_\mu) \Sigma_\nu^{1/2})^{1/2}\Sigma_\nu^{-1/2}O$ is diagonal, $D={\rm diag}\left(\sqrt{\frac{(O^* \Sigma_\nu O)_{11}}{(O^* \cJ_2(\Sigma_\nu,\Sigma_\mu) O)_{11}}},\cdots,\sqrt{\frac{(O^* \Sigma_\nu O)_{dd}}{(O^* \cJ_2(\Sigma_\nu,\Sigma_\mu) O)_{dd}}}\right)$ and $$\cI_2(\Sigma_\mu,\Sigma_\nu)=ODO^*\Sigma_\mu ODO^*.$$
Indeed, $O^*\Sigma_\nu O$ and $O^* \cJ_2(\Sigma_\nu,\Sigma_\mu) O$ share the same correlation matrix by~\eqref{O_share_correl} and $O^*\Sigma_\nu O=D O^* \cJ_2(\Sigma_\nu,\Sigma_\mu) O D \ge D O^* \Sigma_\mu O D$. According to the beginning of the proof of Proposition~\ref{prop_existsO}, $D={\rm diag}\left(1\wedge\sqrt{\frac{(O^*\Sigma_\nu O)_{11}}{(O^*\Sigma_\mu O)_{11}}},\cdots,1\wedge\sqrt{\frac{(O^*\Sigma_\nu O)_{dd}}{(O^*\Sigma_\mu O)_{dd}}}\right)$ so that the equality giving $\cI_2(\Sigma_\mu,\Sigma_\nu)$ holds by Theorem~\ref{thmprojgauss}.

\appendix
 
 \section{Technical results}

\begin{lemma}\label{ordsq}
  Let $M,N\in\cS^+_d$ be such that $M\le N$. Then $M^{1/2}\le N^{1/2}$.
\end{lemma}
For the sake of completeness, we give an elementary proof of this standard result.
\begin{proof}
  Let $\lambda$ be an eigenvalue of $N^{1/2}-M^{1/2}$ and $\xi\in\R^d\setminus\{0\}$ be an associated eigenvector. Since $N-M=\frac 12\left((N^{1/2}-M^{1/2})(N^{1/2}+M^{1/2})+(N^{1/2}+M^{1/2})(N^{1/2}-M^{1/2})\right)$, we have
  \begin{align*}
     0\le \xi^*(N-M)\xi=\lambda \xi^*(N^{1/2}+M^{1/2})\xi.
  \end{align*}
As $\xi^*(N^{1/2}+M^{1/2})\xi\ge \left(\xi^*(N^{1/2}-M^{1/2})\xi\right)\vee \left(\xi^*(M^{1/2}-N^{1/2})\xi\right)=|\lambda||\xi|^2$, we deduce that either $\lambda=0$ or $\lambda>0$.  
\end{proof}
 \begin{lemma}\label{lemconvw2}
  For fixed $\tilde\Sigma\in\cS^+_d$, $\cS^+_d\ni \Sigma\mapsto{\rm tr}\left(\left(\tilde\Sigma^{1/2}\Sigma\tilde\Sigma^{1/2}\right)^{1/2}\right)$ is concave and $\cS^+_d\ni \Sigma\mapsto\cW_2^2(\cN_d(0,\Sigma),\cN_d(0,\tilde \Sigma))=\BW^2(\Sigma,\tilde{\Sigma})$ is convex.
\end{lemma}
\noindent This Lemma is a consequence of~\cite[Theorem 7]{BhJaLi19}, but we give a proof for reader's convenience. 
\begin{proof}
  For $\Sigma_1,\Sigma_2\in\cS^+_d$ and $\alpha\in[0,1]$, we have
  \begin{align*}
     &\left(\alpha(\tilde\Sigma^{1/2}\Sigma_1\tilde\Sigma^{1/2})^{1/2}+(1-\alpha)(\tilde\Sigma^{1/2}\Sigma_2\tilde\Sigma^{1/2})^{1/2}\right)^2=\alpha^2\tilde\Sigma^{1/2}\Sigma_1\tilde\Sigma^{1/2}+(1-\alpha)^2\tilde\Sigma^{1/2}\Sigma_2\tilde\Sigma^{1/2}\\&\;\;\;\;+\alpha(1-\alpha)((\tilde\Sigma^{1/2}\Sigma_1\tilde\Sigma^{1/2})^{1/2}(\tilde\Sigma^{1/2}\Sigma_2\tilde\Sigma^{1/2})^{1/2}+(\tilde\Sigma^{1/2}\Sigma_2\tilde\Sigma^{1/2})^{1/2}(\tilde\Sigma^{1/2}\Sigma_1\tilde\Sigma^{1/2})^{1/2}).
  \end{align*}
  Since for $\xi\in\R^d$, by the Cauchy-Schwarz and Young's inequalities,
\begin{align*}
  \xi^*(\tilde\Sigma^{1/2}\Sigma_1\tilde\Sigma^{1/2})^{1/2}(\tilde\Sigma^{1/2}\Sigma_2\tilde\Sigma^{1/2})^{1/2}\xi&\le \sqrt{\xi^*\tilde\Sigma^{1/2}\Sigma_1\tilde\Sigma^{1/2}\xi}\times\sqrt{\xi^*\tilde\Sigma^{1/2}\Sigma_2\tilde\Sigma^{1/2}\xi }\\&\le\frac 1 2\left(\xi^*\tilde\Sigma^{1/2}\Sigma_1\tilde\Sigma^{1/2}\xi+\xi^*\tilde\Sigma^{1/2}\Sigma_2\tilde\Sigma^{1/2}\xi\right),
\end{align*} we deduce that
$$\left(\alpha(\tilde\Sigma^{1/2}\Sigma_1\tilde\Sigma^{1/2})^{1/2}+(1-\alpha)(\tilde\Sigma^{1/2}\Sigma_2\tilde\Sigma^{1/2})^{1/2}\right)^2\le \alpha\tilde\Sigma^{1/2}\Sigma_1\tilde\Sigma^{1/2}+(1-\alpha)\tilde\Sigma^{1/2}\Sigma_2\tilde\Sigma^{1/2}.$$
By Lemma \ref{ordsq}, we deduce that $$\alpha(\tilde\Sigma^{1/2}\Sigma_1\tilde\Sigma^{1/2})^{1/2}+(1-\alpha)(\tilde\Sigma^{1/2}\Sigma_2\tilde\Sigma^{1/2})^{1/2}\le (\tilde\Sigma^{1/2}\left(\alpha\Sigma_1+(1-\alpha)\Sigma_2\right)\tilde\Sigma^{1/2})^{1/2}.$$ Since $\tr(M)\le\tr(N)$ for $M,N\in\cS^+_d$ such that $M\le N$, we conclude that $\cS^+_d\ni \Sigma\mapsto{\rm tr}\left(\left(\tilde\Sigma^{1/2}\Sigma\tilde\Sigma^{1/2}\right)^{1/2}\right)$ is concave. The convexity of $\cS^+_d\ni \Sigma\mapsto\cW_2^2(\cN_d(0,\Sigma),\cN_d(0,\tilde \Sigma))$ follows by \eqref{w2gauss} and the linearity of the trace.
\end{proof}

\bibliographystyle{abbrv}
\bibliography{joint_biblio}
\end{document}